\pdfoutput=1
\documentclass{article}

\usepackage{PRIMEarxiv}

\usepackage[utf8]{inputenc} 
\usepackage[T1]{fontenc}    
\usepackage{hyperref}       
\usepackage{url}            
\usepackage{booktabs}       
\usepackage{amsfonts}       
\usepackage{nicefrac}       
\usepackage{microtype}      
\usepackage{lipsum}
\usepackage{fancyhdr}       
\usepackage{graphicx}       
\graphicspath{{media/}}     
 
\usepackage{graphicx}
\usepackage{tikz-cd}
\usepackage[utf8]{inputenc}  
\usepackage{amssymb, amsmath}
\usepackage{amsthm}
\usepackage{enumitem}
\usepackage{hyperref}
\hypersetup{
    colorlinks=true,
    linkcolor=blue,
    filecolor=magenta,      
    urlcolor=cyan,
    citecolor=magenta,
    pdfpagemode=FullScreen,
    }

\theoremstyle{remark}
\newtheorem{remark}{Remark}[section]

\theoremstyle{definition}
\newtheorem{theorem}{Theorem}[section]
\newtheorem{lemma}{Lemma}[section]
\newtheorem{corollary}{Corollary}[section]
\newtheorem{proposition}{Proposition}[section]
\newtheorem{definition}{Definition}[section]
\newtheorem{example}{Example}[section]

\newcommand{\cinf}{$\mathcal{C}^{\infty}$}
\newcommand{\contract}{\,\lrcorner\,}

\title{ \cinf-structures in the integration of involutive distributions
\thanks{\textit{\underline{Citation}}: 
\textbf{Authors. Title. Pages.... DOI:000000/11111.}} 
}

\author{
  *A.J. Pan-Collantes, C. Muriel, A. Ruiz, J.L. Romero \\
  Departamento de Matem\'aticas\\ 
  Universidad de C\'{a}diz - UCA \\
  Puerto Real\\
  \texttt{\{antonio.pan@uca.es, concepcion.muriel@uca.es, adrian.ruiz@uca.es, juanluis.romero@uca.es\}} \\
}

\begin{document}

\maketitle

\begin{abstract}
For a system of ordinary differential equations (ODEs) or,  more generally, an involutive distribution  of vector fields, the problem of its integration is considered. Among the many approaches to this problem, solvable  structures provide a systematic procedure of integration via Pfaffian equations that are integrable by quadratures. In this paper structures more general than solvable structures (named \cinf-structures) are considered. The symmetry condition in the concept of solvable structure is weakened for  \cinf-structures by requiring their vector fields be just \cinf-symmetries. For \cinf-structures  there is also an integration procedure, but the corresponding Pfaffian equations, although completely integrable, are not necessarily integrable by quadratures.    
The  well-known result on the relationship between integrating factors and Lie point symmetries for first-order ODEs is generalized for \cinf-structures and  involutive distributions of arbitrary corank by introducing symmetrizing factors.  The role of these symmetrizing factors on the integrability by quadratures of the Pfaffian equations associated with the \cinf-structure is also established.  Some examples that show how these objects and results can be applied in practice are also presented. 

\end{abstract}



\section{Introduction}


Let us consider a first-order differential equation
 \begin{equation}\label{orden1}
 \Delta\equiv   M(x,u)+N(x,u)u_1=0,
\end{equation} where $u_1$ denotes the derivative of the dependent variable $u$ with respect to the independent variable $x.$ Equation \eqref{orden1} can also be written as the total differential equation
\begin{equation}\label{orden1w}
    M(x,u)dx+N(x,u)du=0.
\end{equation}
By setting $\omega= M(x,u)dx+N(x,u)du,$ it is well-known that there always exists a function $\mu=\mu(x,u),$  called an {\it integrating factor}, such that $\mu \omega$ becomes an exact 1-differential form, i.e., there exists a function $F=F(x,u),$ called a {\it primitive}, such that $\mu\omega=dF.$ A primitive can be calculated by a {\it quadrature} as follows \cite{ibragimovlibro}:
$$F(x,u)=\int_{x_0}^x (\mu M)(t,u)dt+\int_{u_0}^u (\mu N)(x_0,t)dt.$$
Then the general solution to \eqref{orden1} can be expressed in implicit form as $F(x,u)=C,$ where $C$ is an arbitrary constant. 
Therefore, any primitive $F$ defines implicitly the {\it integral manifold} (integral curve) of the rank 1 distribution spanned by the vector field $Z=N(x,u)\partial_x-M(x,u)\partial_u$ associated to equation \eqref{orden1}.

The search of integrating factors is normally no easier than the integration of the original equation \eqref{orden1}, although integrating factors of certain prefixed forms can be found in some particular cases. It is also well known that there is a connection between integrating factors and symmetries, established by the Norwegian mathematician S. Lie: if $X=\xi(x,u)\partial_x+\eta(x,u)\partial_u$ is a (Lie point) symmetry of \eqref{orden1}, then the function \begin{equation}\label{muLie}
    \mu=(\xi M+\eta N)^{-1}
\end{equation}
becomes an integrating factor of \eqref{orden1}. This function $\mu=\mu(x,u)$ can be expressed in terms of the  contraction (or interior product)  of the 1-differential form $\omega$ 
with the symmetry $X$ as follows
\begin{equation}\label{mu0}
\mu=\left(X\lrcorner\,\omega\right)^{-1}.\end{equation}


Sherring and Prince proved in  \cite[Theorem 2.1]{sherring1992geometric} that this result can be generalized for a Pfaffian equation defined by a 1-form $\omega$  on some open set $U$ of a manifold of arbitrary dimension $n$ which is Frobenius integrable, i.e., such that $d\omega\wedge \omega=0$. In this general setting, the condition for a vector field $X$ to be  a symmetry of $\omega$ in terms of the Lie derivative $\mathcal{L},$ is $\mathcal{L}_X\omega=\rho\omega,$ for some smooth function (or 0-form) $\rho$ on $U.$ This condition can be equivalently expressed in the form $\mathcal{L}_X \mathcal{D}\subset\mathcal{D},$ where $\mathcal{D}$ denotes the 
distribution  of rank $n-1$ spanned by the vector fields which annihilate $\omega.$   This concept of symmetry can be straightforwardly  extended for systems of $n-r$ Pfaffian equations, or equivalently, for involutive distributions $\mathcal{D}$  of arbitrary rank $r<n$ \cite{lychagin1991,Dubrov1998}. 

In particular, the integrability problem of an $n$th-order system of $m$ ordinary differential equations fits in this geometric setting by considering the  distribution of rank $r=1$  spanned by the vector field $Z$ associated to the differential system. The  extended notion of symmetry, determined in this case by the condition 
\begin{equation}\label{symgeo}
    \mathcal{L}_X Z\in \mathcal{S}(\{Z\}), 
\end{equation}
where $\mathcal{S}(\{Z\}$ denotes the submodule generated by  $Z,$ includes not only standard prolongations of Lie point symmetries but also of generalized Lie symmetries \cite{olver86,ibragimovlibro,stephani,blumanlibro,ovsiannikovlibro,blumankumeilibro} .

Symmetries of distributions were used by Basarab-Howarth to introduce in \cite{basarab} the notion of {\it solvable structure}, which plays a fundamental role in the integrability by quadratures of involutive distributions (see, for instance, \cite[Proposition 3]{basarab} and  \cite[Proposition 4.7]{sherring1992geometric}).  
The solvable structure method of integration recovers, as a particular case,  the classical result that asserts that an $n$th-order system of $m$ ordinary differential equations with an $mn$-dimensional {\it solvable algebra} of point symmetries can be completely integrated by quadratures. However, the procedure based on solvable structures is much more powerful because involves symmetries more general than Lie point symmetries, and computationally is simpler because it avoids canonical coordinates and the associated quotient manifolds (see Section VI in \cite{sherring1992geometric}). Moreover, the elements of a solvable structure, except its first element, are not required to be symmetries of the system of differential equations \cite[Proposition 56]{sherring1992geometric}. 

In \cite{muriel01ima1} it was proved that the traditional point symmetry method of reduction for ODEs can be extended by considering $\lambda$-prolongations \cite[Definition 2.1]{muriel01ima1}, which are more general than standard prolongations of vector fields.  This gave rise to the notion of   $\mathcal{C}^{\infty}$-symmetry (or $\lambda$-symmetry) for ODEs \cite[Definition 2.2]{muriel01ima1}, as  a generalization of Lie point symmetries.   A $\mathcal{C}^{\infty}$-symmetry 
$X$ does not verify the symmetry condition \eqref{symgeo}, but a less restrictive one, namely, 
\begin{equation}
    \mathcal{L}_X Z
    \in \mathcal{S}(\{Z,X\}). 
\end{equation} In contrast to the case of symmetries, a $\mathcal{C}^{\infty}$-symmetry $X$ can be multiplied by any smooth function $f$ and $fX$  is still a $\mathcal{C}^{\infty}$-symmetry \cite[Lemma 5.1]{muriel01ima1}. 

Thanks to the works by authors such as G. Gaeta, P. Morando, G. Cicogna, D. Catalano-Ferraioli, E. Pucci, G. Saccomandi, S. Walcher, among other authors, the original notion of  $\mathcal{C}^{\infty}$-symmetry for ODEs evolved with the years in several directions (see \cite{gaetatwisted,muriel_evolution} and references therein), giving rise to other types of symmetries defined through a deformation of the standard prolongation operation 
\cite{pucci, gaetamorando,gaeta2,murielolver,cicognagaeta_noether,morando_mu_symmetries}.  
Further studies have also shown  the convenience of considering not only the deformation of a single vector field but a of a set of vector fields 
\cite{cicogna2012generalization, cicogna2013dynamical,gaetacollective},   and distributions of vector fields and Frobenius integrability in this type of problems \cite{gaeta_frobenius,PaolaPfaffian,paola.frobenius}. 

The first goal of this paper is to generalize the mentioned result by Sherring and Prince on integrating factors for a single Pfaffian equation $\omega=0$ on a $n$th-dimensional manifold. For this aim, we consider the extended  concept of  $\mathcal{C}^{\infty}$-symmetry for involutive distributions $\mathcal{D}$  of arbitrary rank $r<n$ that has been recently introduced in \cite[Definition 3.2]{pancinf-structures}. When $X$ is  a $\mathcal{C}^{\infty}$-symmetry, the function $(X\,\lrcorner\,\omega)^{-1}$ which appears in \eqref{mu0} is no longer an integrating factor as in the case of symmetries. However, this function $h:=(X\,\lrcorner\,\omega)^{-1},$ which is always well defined, can be used as a ${\it normalizing}$ factor, meaning that the $\mathcal{C}^{\infty}$-symmetry $hX$ satisfies $hX\,\lrcorner\, \omega =1.$ Our first result  proves that integrating factors can be characterized as the reciprocal of symmetrizing factors, i.e.,  functional factors which converts a normalized $\mathcal{C}^{\infty}$-symmetry into a symmetry (see Lemma \ref{lema_inverserelation}). As a consequence, for involutive distributions of corank 1,  we derive
in Theorem \ref{factorgeneral} a  correspondence between integrating factors and symmetrizing factors for arbitrary (not necessarily normalized) $\mathcal{C}^{\infty}$-symmetries.

The following objective addressed in this paper is to  extend these results  to a completely integrable system of $n-r$ Pfaffian equations, which corresponds to a rank $r$ involutive distribution $\mathcal{Z}$ of vector fields. In Theorem \ref{theor_symfactor} we first demonstrate the existence of symmetrizing factors for any $\mathcal{C}^{\infty}$-symmetry of the distribution $\mathcal{Z}.$ Several properties of symmetrizing factors are studied in Corollary \ref{corosymfactor}. Secondly, we investigate the role of symmetrizing factors   in the integrability by quadratures of a distribution $\mathcal{Z}$ of arbitrary rank.  At this point, a single $\mathcal{C}^{\infty}$-symmetry of $\mathcal{Z}$ is not enough for the complete integration; it is necessary to consider an ordered sequence  of $n-r$ vector fields $\{X_1,\ldots,X_{n-r}\}$ that form a specific structure. These structures, introduced in \cite[Definition 3.3]{pancinf-structures} with the name of {\it $\mathcal{C}^{\infty}$-structures}, extend the notion of $\mathcal{C}^{\infty}$-symmetry as solvable structures generalized the notion of symmetry.  A brief summary of the main results on $\mathcal{C}^{\infty}$-symmetries and $\mathcal{C}^{\infty}$-structures of distributions, as well as the main properties that are used throughout  the paper, are presented in Section \ref{preliminaries}.

In Section \ref{secSymFactors} we consider a dual relationship between a \cinf-structure of vector fields for $\mathcal{Z}$ and an associated \cinf-structure of $n-r$ one-forms. We also present a procedure to the stepwise construction of the integral manifolds of $\mathcal{Z}$ by means of the 
successive integration of $n-r$ completely integrable Pfaffian equations, each one defined in a  space of one less dimension \cite[Theorem 3.5]{pancinf-structures}. 
 In this context, a solvable structure is the optimal particular case of $\mathcal{C}^{\infty}$-structure, in which case any of these $n-r$ Pfaffian equations can be integrated by quadrature. However,  solvable structures are more difficult to find in practice than $\mathcal{C}^{\infty}$-structures, because the elements of a solvable structure must satisfy stronger conditions.

 In Section  \ref{sec_4} we first prove that for each element $X_i$ of a $\mathcal{C}^{\infty}$-structure, there exists a symmetrizing factor of $X_i$  with respect to the  distribution $\mathcal{S}(\mathcal{X}_{i-1})$ spanned by the vector fields in $\mathcal{Z}$ and the first $i-1$ vector fields of the \cinf- structure. In  Theorem \ref{relacioninversa2} we establish the role of such symmetrizing factors in the integration procedure associated to the  $\mathcal{C}^{\infty}$-structure.  As a consequence, we derive an indirect method, based on \cinf-structures, to construct solvable structures, which in practice are difficult to find. However, it must be remarked that these solvable structures appear {\it ad hoc}, once the distribution has been completely integrated by the  $\mathcal{C}^{\infty}$-structure method.  
 
 Finally, in Section \ref{sec5} we  present some examples to show how the concepts and results obtained in this paper can be used in practice. In particular, in Example \ref{ex_airy}, the \cinf-structure procedure is used to  integrate completely a third-order ordinary differential equation that cannot be solved by the Lie classical method. 

\section{\texorpdfstring{\cinf}--symmetries 
of distributions and symmetrizing factors}\label{preliminaries} 


Throughout the paper we will work on a star shaped open subset $U$ of $\mathbb{R}^n.$ By $\mathfrak{X}(U)$ and ${\Omega}^k(U)$ we will denote the modules over $\mathcal{C}^{\infty}(U)$ of all smooth vector fields and differential $k$-forms, respectively, while $\Omega^*(U)$ will denote the exterior algebra of all the differential forms on $U$ \cite{Morita,warner}. Given a set $B\subseteq \mathfrak{X}(U)$ (respectively, $B\subseteq \Omega^1(U)$) we will denote by $\mathcal{S}(B)$ the submodule of $\mathfrak{X}(U)$ (respectively, $\Omega^1(U)$) generated by $B$.  
We say that a set of vector fields (or 1-forms) is independent on $U$  if they are pointwise linearly independent at every point of $U.$ 


We will call a rank $r$ \emph{distribution} on $U$ to a submodule of $\mathfrak{X}(U)$ of the form $\mathcal{Z}=\mathcal{S}(\{Z_1,\ldots,Z_r\})$, where $Z_1,\ldots,Z_r$ are $r$ independent vector fields on $U$. In a similar way, a rank $s$ \textit{Pfaffian system} will be a submodule of $\Omega^1(U)$ generated by $s$ independent 1-forms. 
It is well known \cite{warner,bryant2013exterior}  (see \cite{pancinf-structures} and the references therein) that for a distribution $\mathcal{Z}$ of rank $r$ there exists a rank $n-r$ Pfaffian system, which will be denoted by $\mathcal{Z}^{\circ},$ whose elements $\omega\in \mathcal{Z}^{\circ}$ are characterized by the conditions
\begin{equation}
    Z_i\lrcorner\,\omega =0 \quad \mbox{for} \quad 1\leq i\leq r.
\end{equation} 
Conversely, given a rank $n-r$ Pfaffian system ${\Lambda}=\mathcal{S}(\{\omega_1,\ldots,\omega_{n-r}\})$ there exists a rank $r$ distribution, denoted by $\Lambda^{\perp},$ whose elements $Z\in \Lambda^{\perp}$ are characterized by the conditions
\begin{equation}
    Z\lrcorner\,\omega_i =0 \quad \mbox{for} \quad 1\leq i\leq n-r.
\end{equation}  We refer to the distribution $\Lambda^{\perp}$ as the kernel or  the characteristic space  of $\Lambda.$

Let us recall that a distribution $\mathcal{Z}$ is said to be \textit{involutive} if it is closed under the Lie bracket. 
The distribution $\mathcal{Z}$ is involutive if and only if $\mathcal{Z}^{\circ}$ is a {\it completely integrable} Pfaffian system, i.e.,  the algebraic ideal generated by $\mathcal{Z}^{\circ}$ is closed under exterior differentiation (see, for instance, \cite[Proposition 2.30]{warner}). The  (local) existence of maximal integral manifolds for involutive distributions (or for {completely integrable} Pfaffian systems) is guaranteed by the well-known Frobenius Theorem \cite{warner}.

Two important notions concerning the search for integral manifolds of involutive distributions are the concepts of symmetry of a distribution and solvable structure \cite{basarab, hartl1994solvable,lychagin1991, sherring1992geometric,ChrisAthorne1998,Barco2001,Barco2002,BarcoPDE,barco2001similarity}. In the recent paper \cite{pancinf-structures}, generalizations of both concepts were presented, enlarging the class of vector fields that can be used to integrate involutive distributions. 

We first recall the concept of  $\mathcal{C}^{\infty}$-symmetry for distributions, which generalizes the original notion of $\mathcal{C}^{\infty}$-symmetry for a scalar ordinary differential equation \cite{muriel01ima1}:

\begin{definition}[Definition 3.2 in \cite{pancinf-structures}] \label{Csymmdistribution}
Let $\mathcal{Z}=\mathcal{S}(\{ Z_1,\ldots, Z_r \})$ be an involutive distribution  on an open set $U\subseteq \mathbb{R}^n,$ $r<n.$ A vector field $X$ on $U$ such that $\{Z_1,\ldots,Z_{r},X\}$ is linearly independent will be called a {\bf $\mathcal{C}^{\infty}$-symmetry of $\mathcal{Z}$} if the distribution $\mathcal{S}(\{Z_1,\ldots, Z_r, X \})$ is involutive, i.e., for $ i,k=1,\ldots,r$ there exist smooth functions $\lambda_i, c_{ik}$ on $U$ such that
            \begin{equation}\label{deficinfsym}
            [X,Z_i]=\lambda_i X+\sum_{k=1}^r c_{ik} Z_k. 
            \end{equation}
\end{definition}

\begin{remark}
\begin{enumerate}
  \item The standard notion of symmetry of a distribution \cite[Definition 1]{hartl1994solvable} corresponds to the particular case of a \cinf-symmetry where $\lambda_i=0$ for $i=1,\ldots,r.$ 
  \item The original concept of  \cinf-symmetry of a $n$th-oder ODE \cite{muriel01ima1} corresponds to a  \cinf-symmetry of the distribution of rank 1 spanned by the vector field associated to the equation.  
  \item When $r=n-1,$ the independence of the set $\{Z_1,\ldots,Z_{n-1},X\}$ implies that $X$ is a $ \mathcal{C}^{\infty}$-symmetry of the involutive distribution $\mathcal{Z}=\mathcal{S}(\{Z_1,\ldots,Z_{n-1}\}),$ because \eqref{deficinfsym} holds. 
\end{enumerate}
\end{remark}

Let us recall that for ODEs the product of a given $\mathcal{C}^{\infty}$-symmetry by a non-vanishing smooth function is also a $\mathcal{C}^{\infty}$-symmetry (see  \cite[Lemma 5.1]{muriel01ima1}). In the following Proposition, which will be used later,  we show that a similar result holds for $\mathcal{C}^{\infty}$-symmetries of distributions: 

\begin{proposition}
\label{prop_multiplicacion}
Let $\mathcal{Z}=\mathcal{S}(\{ Z_1,\ldots, Z_r \})$ be an involutive distribution of vector fields on an open set $U\subseteq \mathbb{R}^n,$ $r<n,$ and consider a non-vanishing smooth function $h$  on $U.$ Then:
\begin{enumerate}
    \item   If $X$ is a $\mathcal{C}^{\infty}$-symmetry of  $\mathcal{Z}$ then  $hX$ is also a $\mathcal{C}^{\infty}$-symmetry of $\mathcal{Z}.$ 
    \item In particular,  if $Y$ is a symmetry of $\mathcal{Z}$ and $h\in \mathcal{C}^{\infty}(U)$ then $hY$ is  a \cinf-symmetry of $\mathcal{Z}.$ 
\end{enumerate}
\end{proposition}
\begin{proof}
Since $X$ is a $\mathcal{C}^{\infty}$-symmetry of  $\mathcal{Z}$ and $h$ is a non-vanishing function on $U$ then  the set $\{Z_1,\ldots,Z_r,hX\}$ is independent on $U.$ By using \eqref{deficinfsym} and the properties of the Lie bracket, we can write, for $i=1\ldots,r,$
    \begin{equation}\label{eq_obs_cinfsym}
        \begin{array}{lll}
        [hX,Z_i]&=&-Z_i(h)X+h[X,Z_i]=\left(h\lambda_i-Z_i(h) \right)X+\displaystyle\sum_{k=1}^r hc_{ik} Z_k\\ &=&\left(\lambda_i-\dfrac{Z_i(h)}{h} \right)hX+\displaystyle\sum_{k=1}^r hc_{ik} Z_k. 
        \end{array} 
    \end{equation} 
This proves that $hX$ is also a  $\mathcal{C}^{\infty}$-symmetry of $\mathcal{Z}.$ 

The second part is a direct consequence of the previous one.
\end{proof}

The following result follows immediately  from \eqref{eq_obs_cinfsym}:

\begin{corollary}\label{coro_symfactor}
    Let $X$ be a \cinf-symmetry of the involutive distribution $\mathcal{Z}=\mathcal{S}(\{Z_1,\ldots, Z_r\})$ and let $\lambda_1,\ldots\lambda_r$ be the corresponding functions in \eqref{deficinfsym}. A non-vanishing smooth function $f$  on $U$ is such that $fX$ is a symmetry of $\mathcal{Z}$ if and only if \begin{equation}\label{condicionsymfactor}
  {Z_i(f)}=\lambda_if,\quad \mbox{for} \quad i=1,\ldots,r.
\end{equation} 
  In this case the function $f$ will be called a {\bf symmetrizing factor of $X$ with respect to $\mathcal{Z}$}.  
\end{corollary}


The main goal in this section is to investigate the existence of these symmetrizing factors and their role  in the integrability of the distribution. In order to present the results in a more accessible way,  we   begin studying  the case of corank 1 distributions. 

\subsection{Symmetrizing factors and integrating factors for corank 1 distributions}\label{secSymFactors1}

Throughout this section $\omega$ will denote  a Frobenius integrable 1-form which is non-vanishing on some open subset $U$ of $\mathbb{R}^n$, and  $\mathcal{Z}$  will be its characteristic space, i.e., the corank $1$ involutive distribution $\mathcal{S}(\{\omega\})^{\perp}.$
 
The well-known relation \eqref{mu0} between Lie point symmetries and integrating factors for a first-order ODE, i.e. for the Pfaffian equation  \eqref{orden1w} in two variables, was   generalized by Sherring and Prince for a symmetry of the corank 1 distribution $\mathcal{Z}= \mathcal{S}(\{\omega\})^{\perp},$ i.e., for a Pfaffian equation in $n$ variables: 


\begin{theorem}\label{sherring21}\cite[Theorem 2.1]{sherring1992geometric}
 A non-vanishing  smooth function $\mu$ is an integrating factor for $\omega$ if and only if $\mu$ can be written as $\mu=1/(V\,\lrcorner\, \omega),$ for some  symmetry $V$ of  $\mathcal{Z}=\mathcal{S}(\{\omega\})^{\perp}$ with $V \,\lrcorner\, \omega \neq 0$ on $U.$

\end{theorem}

Now we assume that $X$ is not a symmetry of $\mathcal{Z},$ but a \cinf-symmetry. According to \cite[Remark 3.4 (a)]{pancinf-structures}, the function $X\,\lrcorner\, \omega$ does not vanish on $U$ and, therefore, the function $h=(X\,\lrcorner\, \omega)^{-1}$ is well defined on $U.$ 
Although by Theorem \ref{sherring21} $h$ is not an integrating factor of $\omega,$ it makes sense to investigate if there is  a connection between the integrating factors of $\omega$ and the symmetrizing factors of the \cinf-symmetry  $X.$ 

In this section we firstly prove the existence of symmetrizing factors for any \cinf-symmetry of a corank 1 distribution $\mathcal{Z}.$ 
The following proposition is just a reformulation of a result proved by Sherring and Prince in \cite{sherring1992geometric} (see pages 436 and 437) that we state in terms of the notion of \cinf-symmetry:

\begin{theorem}\label{existencia_corango1}
 If $X$ is a \cinf-symmetry of an involutive distribution  $\mathcal{Z}$ of corank 1 then there exists,  at least locally, a  symmetrizing factor of $X$ w.r.t. $\mathcal{Z}.$ 
\end{theorem}

\begin{proof}
   Since $\mathcal{Z}$ is an involutive distribution of corank 1, then  $\mathcal{Z}=\mathcal{S}(\{\omega\})^{\perp}$ where $\omega$ is a Frobenius integrable 1-form. Then there exists an integrating factor $\mu$ for $\omega.$ By Theorem \ref{sherring21}, there exists some symmetry $V$ of $\mathcal{Z},$ {associated to $\mu,$} such that  $V \,\lrcorner\, \omega \neq 0$ on $U.$ Since $\mathcal{Z}$ is of corank 1 and $X \,\lrcorner\, \omega \neq 0,$ then  the symmetry $V$ can be expressed in the form  $V=f X+Z,$  for some non-vanishing smooth function $f$ and some vector field $Z$  in $\mathcal{Z}.$ Therefore $fX=V-Z$ is a symmetry  of $\mathcal{Z},$ which proves that  $f$ is a symmetrizing factor of $X$ w.r.t.  $\mathcal{Z}.$
\end{proof}

Previous theorem guarantees the compatibility of system \eqref{condicionsymfactor} for distributions $\mathcal{Z}$ of corank 1. 

Our next results investigate if   there exists a reciprocal relation between symmetrizing factors of $X$ w.r.t. $\mathcal{Z}$ and integrating factors of  $\omega.$ We begin considering  the case  of a \cinf-symmetry $X$ such that $X \,\lrcorner\, \omega=1:$

\begin{lemma}\label{lema_inverserelation}
Let $X$ be a \cinf-symmetry of $\mathcal{Z}=\mathcal{S}(\{\omega\})^{\perp}$  such that $X\,\lrcorner\, \omega=1.$ A non-vanishing smooth function $f$ is a symmetrizing factor for $X$ w.r.t. $\mathcal{Z}$ if and only if $\mu=1/f$ is an integrating factor for $\omega$.
\end{lemma}
\begin{proof}
We first assume that $f$ is a symmetrizing factor for $X$ w.r.t. $\mathcal{Z}$. Since $fX$ is a symmetry for $\mathcal{Z}$, according to Theorem \ref{sherring21} we have that $1/(fX \,\lrcorner\, \omega)=1/f$
is an integrating factor for $\omega$.

Conversely, if a non-vanishing smooth function $\mu$ is an integrating factor for $\omega$ then Theorem \ref{sherring21} guarantees the existence of a symmetry $V$ of $\mathcal{Z}$ such that $\mu=1/(V\,\lrcorner\,\omega).$ We set $f=V\,\lrcorner\,\omega$. 
Since $\mathcal{Z}$ is of corank 1 and $X\,\lrcorner\,\omega=1\neq 0$ on $U,$ then there must exist a non-vanishing smooth function $\alpha$ and some vector field $Z$  in $\mathcal{Z}$ such that $V=\alpha X+Z.$ From
$$
f=V\,\lrcorner\,\omega=(\alpha X+Z)\,\lrcorner\,\omega=\alpha
$$
it follows that  $fX=V-Z$ and therefore $fX$ is a symmetry of $\mathcal{Z}.$ This proves that  $f=1/\mu$ is a symmetrizing factor for $X$ w.r.t. $\mathcal{Z}$.
\end{proof}


The  result in Lemma \ref{lema_inverserelation} can be easily extended for an arbitrary \cinf-symmetry $X,$ without the restriction $X \,\lrcorner\,\omega=1,$ as follows: 

\begin{theorem}\label{factorgeneral}
Let $\omega$ be a Frobenius integrable 1-form which is non-vanishing on some open subset $U$ of $\mathbb{R}^n$, and denote by $\mathcal{Z}$  the involutive distribution associated to $\mathcal{S}(\{\omega\})$.
Let $X$ be a $\mathcal{C}^{\infty}$-symmetry  w.r.t. $\mathcal{Z}.$ A non-vanishing smooth function $f$ is a symmetrizing factor for $X$ w.r.t. $\mathcal{Z}$ if and only if $\mu=1/ (f X\,\lrcorner\,\omega)$ is an integrating factor for $\omega$.
\end{theorem}
 \begin{proof}
     By Remark 3.4 (a) in \cite{pancinf-structures}, the smooth function $h=( X\,\lrcorner\,\omega)^{-1}$ is  well defined on $U,$ and by Proposition \ref{prop_multiplicacion}, $hX$ is a \cinf-symmetry of $\mathcal{Z}.$ It is clear that $f$ is a symmetrizing factor for $X$ w.r.t. $\mathcal{Z}$ if and only $f/h$ is a symmetrizing factor for $hX$ w.r.t. $\mathcal{Z}.$ Since  the \cinf-symmetry $hX$ satisfies the hypothesis in Lemma \ref{lema_inverserelation}, it follows that $f/h$ is a symmetrizing factor for $hX$ w.r.t. $\mathcal{Z}$ if and only if $h/f=(f X\,\lrcorner\,\omega)^{-1}$ is an integrating factor of $\omega,$ which concludes the proof. 
  \end{proof}

\subsection{Symmetrizing factors for  distributions of arbitrary rank}\label{secSymFactors2}

The results presented in Section \ref{secSymFactors1} for distributions of corank 1 show that the existence of a symmetrizing factor for a \cinf-symmetry provides an integrating factor  for the Pfaffian equation $\omega=0,$ and hence ensures its integratibility by quadrature. Our next goal is to extend these results to study the integrability of  Pfaffian systems with more than one equation. 

With this aim, we first extend the result in Theorem \ref{existencia_corango1} for distributions of arbitrary corank: 

\begin{theorem}\label{theor_symfactor}
Let $\mathcal{Z}=\mathcal{S}(\{ Z_1,\ldots, Z_r\})$ be an involutive distribution on an open set $U\subseteq \mathbb{R}^n$, $r<n,$ and let $X$ be a $ \mathcal{C}^{\infty} $-symmetry for $ \mathcal{Z}$. There exists, at least locally, a \emph{symmetrizing factor} $f$ for $X$ w.r.t. $\mathcal{Z}$. 
\end{theorem}

\begin{proof}
Since $\mathcal{Z}$ is involutive, by  Frobenius Theorem \cite{warner} there exists (locally) a set of functionally independent smooth functions $ \{F_1, \ldots, F_{n-r}\} $ such that 
$$
\mathcal{Z}^{\circ}=\mathcal{S}(\{{dF_1},\ldots,{dF_{n-r}}\}).
$$ 
Since by Definition \ref{Csymmdistribution} the set $\{Z_1,\ldots,Z_r,X\}$ is pointwise linearly independent, 
there must exist some $i_0 \in \{1,\ldots,n-r\}$ such that, locally, $ X(F_{i_0})={dF_{i_0}}(X)\neq 0$ and therefore we can define the non-vanishing function 
$$
f=\frac{1}{X(F_{i_0})}.
$$
Let us prove  that $fX$ is a symmetry of $\mathcal{Z} $. Since $X$ is a \cinf-symmetry of $\mathcal{Z},$ by Definition \ref{Csymmdistribution}, for  $ i=1, \ldots, r,$ we can write
\begin{equation}\label{cuenta}
    \begin{array}{rcl}
\left[fX,Z_i\right]&=&\left[\dfrac{1}{X(F_{i_0})} X,Z_i\right]=
-Z_i \left(\dfrac{1}{X(F_{i_0})}\right) X+\dfrac{1}{X(F_{i_0})} \left[X,Z_i\right]=\\
& = & \dfrac{Z_i(X(F_{i_0}))}{X(F_{i_0})^2} X+\dfrac{1}{X(F_{i_0})} \left( \lambda_i X+\displaystyle\sum_{k=1}^r c_{ik} Z_k \right)=\\
&=& -\dfrac{[X,Z_i](F_{i_0})}{X(F_{i_0})^2} X+ \dfrac{\lambda_i}{X(F_{i_0})} X+\displaystyle\sum_{k=1}^r \dfrac{c_{ik}}{X(F_{i_0})} Z_k =\\
&=&\dfrac{\lambda_i X(F_{i_0}) -[X,Z_i](F_{i_0})}{X(F_{i_0})^2} X+\displaystyle\sum_{k=1}^r \dfrac{c_{ik}}{X(F_{i_0})} Z_k,
\end{array}
\end{equation} where the functions $c_{ik}$ are given in \eqref{deficinfsym}.

Since 
$$
[X,Z_i](F_{i_0})=\left(\lambda_i X+\sum_{k=1}^r c_{ik} Z_k\right)(F_{i_0})=\lambda_i X(F_{i_0}),
$$
it follows from \eqref{cuenta} that
$$
\left[fX,Z_i\right]=\sum_{k=1}^r \dfrac{c_{ik}}{X(F_{i_0})} Z_k,
$$
which proves that $ f X $ is a symmetry of $ \mathcal{Z}$.
\end{proof}

We conclude this section with the following statements concerning $\mathcal{C}^{\infty}$-symmetries and symmetrizing factors, which will be useful for further reference: 

\begin{corollary}\label{corosymfactor}
Let $\mathcal{Z}=\mathcal{S}(\{ Z_1,\ldots, Z_r \})$ be an involutive distribution on an open set $U\subseteq \mathbb{R}^n.$ 
    \begin{enumerate}[label=(\arabic*)]
        \item If $X$ is  a $\mathcal{C}^{\infty}$-symmetry of  $\mathcal{Z}$ then there exists a symmetry $Y$ of $\mathcal{Z}$ and a function $h$ such that $X=hY$.
        
        \item For   a $\mathcal{C}^{\infty}$-symmetry $X$ of  $\mathcal{Z},$ the corresponding system \eqref{condicionsymfactor} is compatible. 
        
        \item 
        Let $X$ be  a $\mathcal{C}^{\infty}$-symmetry of  $\mathcal{Z}$ with a given symmetrizing factor $f.$ A function $\widetilde{f}$ is a symmetrizing factor of $X$ w.r.t. $\mathcal{Z}$ if and only if $\widetilde{f}=fg$ for some first integral $g$ of the distribution $\mathcal Z$. 
    \end{enumerate}
\end{corollary}

\begin{proof}

    \begin{enumerate}
        \item[(1)] Let $f$ be a non-vanishing symmetrizing factor of $X$ w.r.t $\mathcal{Z}$ provided by Theorem \ref{theor_symfactor}. The result follows by taking  $h=1/f.$ 
    
        \item[(2)] It is direct a consequence of Corollary \ref{coro_symfactor} and Theorem \ref{theor_symfactor}. 
        
        \item[(3)] If $f$ and $\widetilde{f}$ are two symmetrizing factors of $X$ w.r.t. $\mathcal{Z}$ then by \eqref{condicionsymfactor}  the function $g=\widetilde{f}/{f}$ satisfies
        $$
        Z_i(g)=Z_i\left(\frac{\widetilde{f}}{f}\right)=\frac{Z_i(\widetilde{f})f-\widetilde{f}Z_i(f)}{f^2}=\frac{\widetilde{f}}{f}\left(\frac{Z_i(\widetilde{f})}{\widetilde{f}}-\frac{Z_i(f)}{f} \right)=0 \text{, for } i=1,\ldots,r.
        $$ 
        Conversely, if $f$ is a symmetrizing factor of $X$ w.r.t. $\mathcal{Z}$  and a function $g$ satisfies $Z_i(g)=0$ for $1\leq i\leq r,$ then by \eqref{condicionsymfactor} we can write
        $$ 
        \lambda_i-\frac{Z_i(gf)}{gf}= \lambda_i-\frac{g Z_i(f)}{gf} =\lambda_i-\frac{Z_i(f)}{f}=0\text{, for } i=1,\ldots,r,  
        $$
        which proves that $\widetilde{f}=gf$ is a symmetrizing factor of $X$ w.r.t. $\mathcal{Z}.$ 
    \end{enumerate}
\end{proof}
\section{\texorpdfstring{\cinf}--structures and integrability of distributions}\label{secSymFactors}

The existence of symmetrizing factors of a  \cinf-symmetry $X,$ w.r.t. an involutive distribution of arbitrary corank,  has been proved in Theorem \ref{theor_symfactor}. In particular, when the corank is 1,  it has been shown that the knowledge of a symmetrizing factor for a  \cinf-symmetry $X$ allows to integrate the distribution by quadrature (Lemma \ref{lema_inverserelation} and Theorem \ref{factorgeneral}). In the following  section we aim to study the role of symmetrizing factors in the integrability by quadratures of a distribution of arbitrary  corank $n-r>1.$ In order to do that, we  need not only a single $\mathcal{C}^{\infty}$-symmetry of $\mathcal{Z}$ but an ordered sequence  of $n-r$ vector fields $\{X_1,\ldots,X_{n-r}\}$ that form a specific structure (introduced in \cite{pancinf-structures} with the name of \cinf-structure). For the sake of completeness, in this section we  present the main notions and results on \cinf-structures that are relevant for a better understanding of what follows.
\subsection{The notion of \texorpdfstring{\cinf}--structure}
A natural extension of the concept of solvable structure \cite{basarab,hartl1994solvable} can be attained when we consider \cinf-symmetries instead of standard symmetries: 



\begin{definition}[Definition 3.3 in \cite{pancinf-structures}]\label{Qsolvable}
Let $ \mathcal{Z}=\mathcal{S}(\{Z_1,\ldots,Z_r\}) $ be an involutive distribution on an open subset $U$ of $\mathbb{R}^n$. Let $ \langle X_1,\ldots, X_{n-r}\rangle $ be an ordered collection of vector fields on $U$ and let us define the following sets:  $\mathcal{X}_0:=\{Z_1,\ldots,Z_{r}\}$ and
$$
\mathcal{X}_k:=\{Z_1,\ldots,Z_{r},X_1,\ldots,X_{k}\}, \quad  \mbox{for} \quad 1\leq k\leq n-r.
$$
We will say that $ \langle X_1,\ldots, X_{n-r}\rangle $ is a {\bf \cinf-structure for $\mathcal{Z}$} if, for $1\leq k\leq n-r$, the vector field $X_k$ is a $ \mathcal{C}^{\infty}$-symmetry of the rank $r+k-1$ distribution $\mathcal{S}(\mathcal{X}_{k-1}).$
\end{definition}

\begin{remark}
\begin{enumerate}
    \item A solvable structure  a particular case of \cinf-structure in which every $X_k$ is a symmetry of $\mathcal{S}(\mathcal{X}_{k-1}),$ for $1\leq k\leq n-r$ \cite{basarab, hartl1994solvable,lychagin1991, sherring1992geometric,ChrisAthorne1998,Barco2001}. 
   \item As a direct consequence of Proposition \ref{prop_multiplicacion},  any of the elements $X_i$ of a \cinf-structure can be multiplied by a non-vanishing function $h$ in such a way that $\langle X_1,\ldots, hX_i,\ldots, X_{n-r}\rangle$ is also a  \cinf-structure for $\mathcal{Z}.$
\end{enumerate}
\end{remark}





From a dual point of view, the notion of \cinf-structure can be expressed in terms of 1-forms  as follows: 

\begin{definition}[Definition 3.4 in \cite{pancinf-structures}]\label{Qsolvabledual}
Let $\mathcal{Z}=\mathcal{S}(\{Z_1,\ldots,Z_r\})$ be an involutive distribution on an open set $U\subset\mathbb{R}^n$. Let $\langle \omega_1,\ldots, \omega_{n-r}\rangle$ be an ordered collection of linearly independent 1-forms on $U$ and define the sets
\begin{equation}\label{wk}
    \Lambda_k:=\{ \omega_{k+1},\ldots, \omega_{n-r}\}\quad\mbox{for}\quad 0\leq k\leq n-r-1\quad \mbox{and} \quad \Lambda_{n-r}=\emptyset.
\end{equation} 
We will say that $\langle \omega_1,\ldots, \omega_{n-r}\rangle$ is a {\bf
$\mathcal{C}^{\infty}$-structure of 1-forms} for $ \mathcal{Z} $ if the following two conditions hold:
\begin{enumerate}[label=(\alph*)]
\item $\mathcal{S}(\Lambda_0)=\mathcal{Z}^{\circ}$.
\item $\mathcal{I}(\Lambda_k) $ is a differential ideal for $k=1,\ldots, n-r-1$.
\end{enumerate}
\end{definition}

In \cite[Proposition 3.3]{pancinf-structures} it was established a correspondence between the notions introduced in definitions \ref{Qsolvable} and  \ref{Qsolvabledual}: given a \cinf-structure $\langle X_1,\ldots, X_{n-r}\rangle$  for $\mathcal{Z}$, there exists a \cinf-structure of 1-forms $\langle \omega_1,\ldots, \omega_{n-r}\rangle$ for $\mathcal{Z}$ such that
\begin{equation}\label{eq_condiciondual}
    \mathcal{S}\left( \mathcal{X}_{k} \right)^{\circ}=\mathcal{S}\left( \Lambda_{k} \right),\quad\mbox{for}\quad 0\leq k\leq n-r-1.  
\end{equation} Conversely,  given a \cinf-structure of 1-forms $\langle \omega_1,\ldots, \omega_{n-r}\rangle$  for $\mathcal{Z}$ there exists  a \cinf-structure $\langle X_1,\ldots, X_{n-r}\rangle$  for $\mathcal{Z}$  such that
\begin{equation}\label{eq_condiciondual2}
   \mathcal{S}\left( \Lambda_{k} \right)^{{\perp}}= \mathcal{S}\left( \mathcal{X}_{k} \right),\quad\mbox{for}\quad 0\leq k\leq n-r-1.  
\end{equation}
If any of the two equivalent conditions \eqref{eq_condiciondual} or \eqref{eq_condiciondual2} holds,  we will say that $\langle X_1,\ldots, X_{n-r}\rangle$  and   $\langle \omega_1,\ldots, \omega_{n-r}\rangle$ are {\bf $\mathcal{Z}$-related}. 
\subsection{Construction of \texorpdfstring{$\mathcal{Z}$}--related  \texorpdfstring{\cinf}--structures}
In practice, a procedure to construct $\mathcal{Z}$-related \cinf-structures can be carried out as follows. Let  $\{x_1,\ldots,x_n\}$ be a local system of coordinates on $U$ and set the volume form $\boldsymbol{\Omega}=dx_1\wedge dx_2\wedge \ldots \wedge dx_n.$ 

Given a \cinf-structure $\langle X_1,\ldots, X_{n-r}\rangle$  for $\mathcal{Z}$, we consider the 1-forms 
    \begin{equation}\label{contractedforms}
            \omega_i:={X_{n-r}\,\lrcorner\,\ldots\,\lrcorner\,\widehat{X_i}\,\lrcorner\,\ldots\,\lrcorner\,X_{1}\contract Z_{r}\contract \ldots\,\lrcorner\,Z_{1}\,\lrcorner\,\boldsymbol{\Omega}}, \quad \mbox{for} \quad 1\leq i\leq n-r,
    \end{equation} 
    where $\widehat{X_i}$ indicates omission of $X_i.$ 
Then it can be checked that $\langle \omega_1,\ldots, \omega_{n-r}\rangle$ is a  $\mathcal{Z}$-related  \cinf-structure of 1-forms. 
\begin{remark}
    We observe that by  the independence of the set $\mathcal{X}_{n-r},$ the determinant
\begin{equation}\label{delta}
    \Delta={X_{n-r}\,\lrcorner\,\ldots\,\lrcorner\,X_{1}\contract Z_{r}\contract \ldots\,\lrcorner\,Z_{1}\,\lrcorner\,\boldsymbol{\Omega}}=\left|\begin{array}{ccc} Z_1^1 & \ldots & Z_1^n\\ 
   \vdots &   \vdots & \vdots  \\
    Z_r^1 & \ldots & Z_r^n\\ 
    X_1^1 & \ldots & X_1^n\\ 
     \vdots &   \vdots & \vdots  \\
    X_{n-r}^1 & \ldots & X_{n-r}^n\\\end{array}\right|
\end{equation} does not vanish on $U.$ Therefore the 1-forms 
\begin{equation}\label{contractedformsnormalizadas}
\sigma_i:= \dfrac{(-1)^{n-r-i}}{\Delta}\omega_i \end{equation} where the $\omega_i$ are defined in \eqref{contractedforms},  are well defined on $U$ and they  define a new $\mathcal{Z}$-related  \cinf-structure of 1-forms satisfying the conditions
\begin{equation}\label{dualidad} 
    X_i\,\lrcorner\,\sigma_j=\delta_{ij}, \quad i,j=1,\ldots, n-r.
\end{equation} Any pair of $\mathcal{Z}$-related \cinf-structures  $\langle X_1,\ldots, X_{n-r}\rangle$  and   $\langle \sigma_1,\ldots, \sigma_{n-r}\rangle$ satisfying \eqref{dualidad} will be called {\bf dually $\mathcal{Z}$-related}. 
\end{remark}

\subsection{Procedure of integration by \texorpdfstring{\cinf}--structures}\label{algorithm}

The notion of \cinf-structure introduced in the previous sections plays a fundamental role in the determination of the integral manifolds of the distribution $\mathcal{Z}.$ The elements of a known  \cinf-structure of 1-forms for $\mathcal{Z}$ permit the stepwise construction of such integral manifolds through the integration of $n-r$ completely integrable Pfaffian equations \cite[Theorem 3.5]{pancinf-structures}.

The main ideas of such constructive process can be organized through the following procedure: 
\begin{enumerate}
    \item Given a \cinf-structure  $\langle X_1,\ldots, X_{n-r}\rangle$  for $\mathcal{Z}$, construct the $\mathcal{Z}$-related \cinf-structure of 1-forms $\langle \omega_1, \ldots,\omega_{n-r}\rangle$ defined by \eqref{contractedforms}. 

  For $1\leq i\leq n-r$ and  $1\leq j\leq n,$ let $\Delta_{r+i,j}$ denote  the determinant formed by deleting the $(r+i)$th row and the $j$th column in \eqref{delta}. Then it can be checked that the expression of the 1-form $\omega_i$ defined in \eqref{contractedforms} becomes
  $\omega_i=\omega_i^1dx_1+\ldots+\omega_i^ndx_n,$ where 
 $\omega_i^j=(-1)^{n+j}\Delta_{r+i,j}$ for $1\leq j\leq n.$ 
    
    
    \item \label{it_solvePfaf} Solve the completely integrable Pfaffian equation $\omega_{n-r}\equiv 0,$ obtaining a smooth function $I_{n-r}$ on an open set $U_{n-r}\subseteq U\subseteq \mathbb{R}^n,$ which defines the  corresponding integral manifolds in the form $\Sigma_{(C_{n-r})}=\{x\in U_{n-r}: I_{n-r}(x)=C_{n-r}\},$ where $C_{n-r}\in \mathbb{R}.$ 
    
    \item Let $\iota_{n-r}:N_{n-r-1}\to \Sigma_{(C_{n-r})},$ where $N_{n-r-1}$ is an open subset of $\mathbb{R}^{n-1},$  be a local  parametrization of $\Sigma_{(C_{n-r})}.$  For  $1\leq i\leq n-r,$ set $\widetilde{\omega}_{i}:=\iota_{n-r}^*(\omega_i)$ and denote by $\widetilde{\mathcal{Z}}$ the restriction of ${\mathcal{Z}}$ to  $N_{n-r-1},$ i.e., $\widetilde{\mathcal{Z}}=\mathcal{S}(\{{(\iota_{n-r}})_*(Z_1),\ldots,({\iota_{n-r}})_*(Z_r)\}).$
    
    \item Since  $\langle \widetilde{\omega}_1,\ldots,\widetilde{\omega}_{n-r-1}\rangle$
    is a \cinf-structure of 1-forms w.r.t.  $\widetilde{\mathcal{Z}}$ 
    we can move back to step \ref{it_solvePfaf}, and solve the completely integrable Pfaffian equation $\widetilde{\omega}_{n-r-1}\equiv 0.$ Observe that this Pfaffian equation is defined on  $N_{n-r-1}\subset\mathbb{R}^{n-1},$ i.e., the dimension has been lowered by one, although the equation depends on the constant $C_{n-r}.$ 
    
   \item This process can be continued and at any intermediate step we have a \cinf-structure of 1-forms defined on an open set $N_k \subseteq \mathbb{R}^{r+k}.$ Such   \cinf-structure has  $k$ elements that will be simply denoted by $\langle \widehat{\omega}_1,\ldots,\widehat{\omega}_{k}\rangle,$ in order to not complicate the notation. 
   Now the integration of the completely integrable Pfaffian equation
    $\widehat{\omega}_{k}\equiv 0
    $
   provides a smooth function 
   $I_k$ on some open set $U_k\subseteq \mathbb{R}^{r+k}$ which defines the corresponding integral manifolds: $$\Sigma_{(C_k,C_{k+1},\ldots,C_{n-r})}=\{x\in U_k: I_k(x)=C_k\}.$$ As before, we consider a local parametrization $\iota_k:N_{k-1}\to \Sigma_{(C_k,C_{k+1},\ldots,C_{n-r})},$  defined on an open set $N_{k-1}$  of $\mathbb{R}^{r+k-1},$ 
   and  the procedure can be continued.
   

    \item After $n-r$ steps,  we get a local parametrization $\iota_1:N_0\to \Sigma_{(C_1,C_2,\ldots,C_{n-r})},$ defined on some open set $N_0$  of $\mathbb{R}^r,$ of the corresponding integral manifold of the Pfaffian equation $\widehat{\omega}_1=0.$ Then the integral manifolds of $\mathcal{Z}$ are given by
     $$
    \iota_{n-r}\circ \ldots \circ \iota_2 \circ \iota_1(N_0),
    $$
   which depend on $(n-r)$ constants $C_{1}, \ldots,C_{n-r}$.

\end{enumerate}
\begin{remark}
    A relevant observation is that, in the previous procedure, only the function $I_{n-r}$ is in fact a first integral of $\mathcal{Z}$ but, for $1\leq k \leq n-r-1,$  the functions $I_k,$  defined  on a space of dimension $r+k,$ are not. 
\end{remark} 

\subsection{Example}

In this subsection we include an illustrative example to show how the previous procedure can be used to integrate an involutive distribution of two vectors field on $\mathbb{R}^4.$

\begin{example}\label{ex_distribucion}

Let us consider the distribution $\mathcal{Z}=\mathcal{S}(\{Z_1,Z_2\})$ on $\mathbb{R}^4$ generated by the vector fields
\begin{equation}\label{Z_ex1}
    \begin{array}{l}
    Z_1=(x_2-x_3 x_4) \partial_{x_2}-x_3 \partial_{x_3}+x_4 \partial_{x_4},\quad
    Z_2=\partial_{x_1}-x_3^2 \partial_{x_3}+(2 x_3 x_4-x_2)\partial_{x_4}.
\end{array}
\end{equation}

The distribution $\mathcal{Z}$ is involutive   because $[Z_1,Z_2]=-x_3 Z_1$. In this example we firstly determine a \cinf-structure for $\mathcal{Z},$ which will be  used to integrate  completely the given distribution. 

As a first element $X_1$ of a \cinf-structure for $\mathcal{Z}$
we can choose the vector field $X_1=x_4 \partial_{x_2}+\partial_{x_3},$ because the commutation relationships of $X_1$ with $Z_1$ and $Z_2$ become
$$
\begin{array}{l}
[X_1,Z_1]=-X_1, \quad
\left[X_1,Z_2\right]=Z_1-x_3 X_1.
\end{array}
$$
According to  \cite[Remark 3.2 (d)]{pancinf-structures},  the second element $X_2$ of a \cinf-structure for $\mathcal{Z}$ in some open set $U\subseteq \mathbb{R}^4$ can be any vector field such that the set $\{Z_1, Z_2,X_1,X_2\}$ is linearly independent on $U$. If we  choose, for instance,  $X_2=\partial_{x_4},$ then 
\begin{equation}\label{det_ex1}
    \Delta={X_{2}\,\lrcorner\,X_{1}\,\lrcorner\, Z_{2} \,\lrcorner\,Z_{1}\,\lrcorner\,\boldsymbol{\Omega}}=-x_2,
\end{equation} and hence the ordered set   $\langle X_1,X_2\rangle$ where \begin{equation}\label{cinf_ex1}
      X_1= x_4 \partial_{x_2}+\partial_{x_3},\quad X_2=\partial_{x_4},
  \end{equation} is a \cinf-structure for $\mathcal{Z}$ on the open set $U=\{(x_1,x_2,x_3,x_4)\in \mathbb{R}^4: x_2\neq 0\}.$
We can integrate $\mathcal{Z}$ by following the steps described at the beginning of this Section \ref{algorithm}:
\begin{enumerate}
    \item We use $\langle X_1,X_2\rangle$ to  construct the $\mathcal{Z}$-related \cinf-structure of 1-forms $\langle \omega_1,\omega_2\rangle$ defined by \eqref{contractedforms}. The coordinates of such 1-forms can be calculated from the corresponding determinant in \eqref{delta} as it was explained in the first step of the integration procedure: 
\begin{equation}\label{omega1_ex}  
    \omega_1=x_3^2(x_2-x_3 x_4)dx_1+{x_3} dx_2+(x_2-x_3 x_4)dx_3.
\end{equation}
Similarly, 
\begin{equation}\label{omega2_ex}
   \omega_2=
-(x_2-x_3x_4)^2dx_1+{x_4} dx_2-{x_4^2}dx_3-x_2dx_4.
\end{equation}
\item In the second step we have to solve the completely integrable  Pfaffian equation
$
\omega_2\equiv 0.
$  The corresponding integral manifolds are implicitly defined in terms of a smooth function  $I_2=I_2(x_1,x_2,x_3,x_4)$ satisfying the condition $dI_2\wedge\omega_2=0.$ It can be checked that this condition leads to the following system of functionally independent determining equations for $I_2:$
\begin{equation}\label{system_ex1}
    \begin{array}{lll}
                             x_2(I_2)_{x_2}+x_4  (I_2)_{x_4}  & = &0, \\
                              x_2(I_2)_{x_3}-x_4^2  (I_2)_{x_4}  & = &0, \\  
                              x_2(I_2)_{x_1}-(x_2-x_3x_4)^2  (I_2)_{x_4}  & = &0.\\
    \end{array}
\end{equation} The general solution of the first equation in \eqref{system_ex1} becomes 
\begin{equation}\label{formaI2}
    I_2=F_0\left(x_1,x_3,\dfrac{x_4}{x_2}\right).
\end{equation} From \eqref{system_ex1} we obtain that  the smooth function $F_0=F_0(a,b,c)$ must satisfy the system 
\begin{equation}\label{system2_ex1}
    \begin{array}{lll}
                              (F_0)_b-c^2   (F_0)_c  & = &0, \\  
                               (F_0)_a-(1-bc)^2   (F_0)_c  & = &0.
    \end{array}
\end{equation} From the first equation in \eqref{system2_ex1} we get 
$$F_0=F_1\left(a,\dfrac{1}{c}-b\right),$$ where $F_1=F_1(r,t)$ is any solution to the equation $$(F_1)_r+t^2(F_1)_t=0.$$ A particular solution becomes $F_1(r,t)=r+1/t,$ which leads to the particular solution $F_0(a,b,c)=a+\dfrac{c}{1-bc}$ of system \eqref{system2_ex1}. By \eqref{formaI2}, the corresponding particular solution for system \eqref{system_ex1} becomes: 
\begin{equation}\label{I2sol}
    I_2=
    x_1+\dfrac{x_4}{x_2-x_3 x_4},
\end{equation} which is defined on the open set $U_2=\{(x_1,x_2,x_3,x_4)\in U: x_2-x_3 x_4\neq 0\}.$
The function given in \eqref{I2sol} is a first integral for the involutive distribution $\mathcal{S}(\mathcal{X}_1)=\mathcal{S}(\{Z_1,Z_2,X_1\}).$ 

In consequence, the integral manifolds of $\omega_2\equiv 0$ can be locally defined as  $$\Sigma_{(C_2)}=\{(x_1,x_2,x_4,x_4)\in U_2: x_1+\dfrac{x_4}{x_2-x_3 x_4}=C_2\},\quad C_2\in \mathbb{R}.$$ 

\item A local parametrization of $\Sigma_{(C_2)}$ is given, for instance, by the function $\iota_2:N_1\to\Sigma_{(C_2)}$  defined by 
$$\iota_2(x_1,x_2,x_3)=\left(x_1,x_2,x_3,\dfrac{x_2(C_2-x_1)}{1+x_3(C_2-x_1)}\right)$$
on the open set 
$$N_1=\{(x_1,x_2,x_3)\in \mathbb{R}^3: 1+x_3(C_2-x_1)\neq0\}.$$
The pullback  by $\iota_2$ of the 1-form $\omega_1$ given in \eqref{omega1_ex} becomes 
\begin{equation}\label{omega1_ex1}
    \widetilde{\omega}_1:=\iota_2^*(\omega_1)=\frac{x_2x_3^2}{1+x_3(C_2 -x_1)} dx_1+{x_3}dx_2+\frac{{x_2}}{1+x_3(C_2 -x_1)}dx_3,
\end{equation}

and defines a \cinf-structure of 1-forms $\langle \widetilde{\omega}_1 \rangle$ for the restriction of $\mathcal{Z}$ to $N_1\subset\mathbb{R}^3.$

\item The Pfaffian equation $\widetilde{\omega}_1\equiv 0$ is completely integrable. From the condition $dI_1\wedge \widetilde{\omega}_1=0$ we get the following system of determining equations for $I_1=I_1(x_1,x_2,x_3;C_2):$

\begin{equation}\label{system3_ex1}
    \begin{array}{lll}
                             (I_1)_{x_1}+x_3  (I_1)_{x_3}  & = &0, \\
                              x_2(I_1)_{x_2}-x_3(1+x_3(C_2-x_1))  (I_1)_{x_3}  & = &0.\\  
    \end{array}
\end{equation}
From the first equation in \eqref{system3_ex1} we get
$$I_1=G_0\left(x_2, \dfrac{1}{x_3}-x_1\right),$$ where, by the second equation in \eqref{system3_ex1}, the function $G_0=G_0(a,b)$ must satisfy the equation 
$$aG_a+(b+C_2)G_b=0.$$ A particular solution for this last equation becomes $G(a,b)=\dfrac{C_2+b}{a},$ giving rise to the following particular solution for system \eqref{system3_ex1}: 
\begin{equation}\label{I1_ex1}
    I_1=\dfrac{1+x_3(C_2-x_1)}{x_2x_3},
\end{equation} defined on the open set 
$U_1=\{(x_1,x_2,x_3)\in N_1: x_2x_3\neq0\}.$ In consequence, the integral manifolds of $\widetilde{\omega}_1\equiv 0$ become
$$\Sigma_{(C_1,C_2)}=\{(x_1,x_2,x_3)\in U_1: \dfrac{1+x_3(C_2-x_1)}{x_2x_3}=C_1\},\quad C_1,C_2\in \mathbb{R}.$$ 
\item For $\Sigma_{(C_1,C_2)}$ we can choose the local parametrization $\iota_1:N_0\to \Sigma_{(C_1,C_2)}$  given by the function $$\iota_1(x_1,x_2)=\left(x_1,x_2,\frac{1}{x_1+C_1 x_2-C_2} \right)$$ defined on the open set $N_0=\{(x_1,x_2)\in \mathbb{R}^2: x_1+C_1 x_2-C_2\neq 0\}.$
\item In order to obtain the 2-dimensional integral manifolds of $\mathcal{Z}$ we first calculate $\iota_2\circ\iota_1.$ For $(x_1,x_2)\in N_0$ and $C_1\neq 0,$ we have
$$\iota_2\circ\iota_1(x_1,x_2)=
\left( x_1,x_2,\frac{1}{x_1+C_1 x_2-C_2} , \dfrac{(C_2-x_1)(x_1+C_1 x_2-C_2)}{C_1}\right).$$
Therefore the integral manifolds for the distribution $\mathcal{Z}$ are locally parametrized by $\iota_2\circ\iota_1$ on $N_0.$ Alternatively, such integral manifolds can be implicitly described by the equations:
\begin{equation}
    x_3=\frac{1}{x_1+C_1 x_2-C_2} , \quad x_4=\dfrac{(C_2-x_1)(x_1+C_1 x_2-C_2)}{C_1}.
\end{equation}
\end{enumerate}
\begin{remark}
     A possible variant of the previous procedure is the following. 
We first  observe that only $\omega_2$ is necessary to carry out the step 2.  The 1-form $\omega_1$ calculated in \eqref{omega1_ex} has been used in the step 3 to determine the 1-form \eqref{omega1_ex1} whose Pfaffian equation is integrated in the fourth step. An alternative way to construct such Pfaffian equation proceeds along the next lines: 
    \begin{enumerate}
        \item Step 1: Calculate only $\omega_2,$ as in \eqref{omega2_ex}, from a known \cinf-symmetry $X_1$ of the distribution $\mathcal{Z}$ (the knowledge of the last element $X_2$ of the \cinf-structure  is not necessary). 
        \item Step 2: Integrate the Paffian equation $\omega_2\equiv 0,$ as in the previous procedure.
        \item Step 3: Restrict the distribution $\mathcal{Z}$ to $N_2$ through $({\iota_2})_*:$ $\widetilde{\mathcal{Z}}=\mathcal{S}(\{{(\iota_{2}})_*(Z_1),({\iota_{2}})_*(Z_2)\}).$ Then the Pfaffian equation $\widetilde{\omega}_1\equiv 0$ arises from any \cinf-structure of 1-forms $\langle \widetilde{\omega}_1\rangle$ for $\widetilde{\mathcal{Z}}.$
    \end{enumerate}
\end{remark}
\end{example}

\section{Relative  symmetrizing factors and integrating  factors}\label{sec_4}

When for an involutive distribution $\mathcal{Z}=\mathcal{S}(\{Z_1\ldots,Z_r\})$ we know not only a single \cinf-symmetry, but a \cinf-structure $\langle X_1,\ldots,X_{n-r}\rangle,$ then Theorem \ref{theor_symfactor} demonstrates, for $1\leq i\leq n-r,$ the existence of a symmetrizing factor for $X_i$ w.r.t. $\mathcal{S}(\mathcal{X}_{i-1})$ (or  a $\mathcal{S}(\mathcal{X}_{i-1})$-symmetryzing factor for short).  In what follows we investigate the role of these symmetrizing factors in the integrability of the  Pfaffian equations arising from the application of the procedure explained in Section \ref{algorithm}.

First we extend the result given in Lemma  \ref{lema_inverserelation} to \cinf-structures of distributions of arbitrary corank. With this aim we first  need the following definition:

\begin{definition}
Let $\langle \omega_1,\ldots,\omega_{n-r}\rangle$ be a \cinf-structure of 1-forms for an involutive distribution $\mathcal{Z}$. Given ${i_0}\in\{1,\ldots,n-r\}$ we will say that a non-vanishing smooth function $\mu$ on $U$ is a \emph{integrating factor of $\omega_{i_0}$ with respect to $\mathcal{S}(\Lambda_{i_0})$} (or a $\mathcal{S}(\Lambda_{i_0})$-integrating factor for short)  if $\mu\omega_{i_0}$ is closed module $\Lambda_{i_0}$, i.e., if $d(\mu\omega_{i_0})\in \mathcal{I}(\Lambda_{i_0}).$ 
\end{definition}

\begin{remark}\label{relativeIF}
\begin{enumerate}
\item Since $\Lambda_{n-r}=\emptyset,$  the condition $d(\mu\omega_{n-r})\in \mathcal{I}(\Lambda_{n-r})$ must be understood as $d(\mu\omega_{n-r})=0,$ i.e., in this case $\mu$ is in fact an integrating factor for $\omega_{n-r}.$

    \item Although in the previous definition the term {\it {integrating factor}} has been used, a relevant observation is that for $i_0<n-r$, the 1-form $\omega_{i_0}$ is not, in general, Frobenius integrable, and therefore an integrating factor for $\omega_{i_0}$ does not exist. However, if $\mu$ is an integrating factor  of $\omega_{i_0}$ w.r.t. $\mathcal{S}(\Lambda_{i_0}),$  then the restriction of $d(\mu\omega_{i_0})$ to the integral manifolds of the involutive distribution $\mathcal{S}(\mathcal{X}_{i_0})$ must vanish, because $\mathcal{S}( \mathcal{X}_{i_0})^{\circ} = \mathcal{S}(\Lambda_{i_0}).$ In consequence, the knowledge of the function $\mu$ yields the integration by quadrature of the corresponding Pfaffian equation $\widehat{\omega}_{i_0}\equiv 0$ stated in the fifth step of the procedure described in Section \ref{algorithm}.  On the other hand, for $i_0\neq n-r,$ an integrating factor  of $\widehat{\omega}_{i_0}$  
is a smooth function defined on a space of dimension $r+i_0<n,$ whereas a $\mathcal{S}(\Lambda_{i_0})$-integrating factor must be a smooth function defined on some open set of $\mathbb{R}^n.$ 

\item Given a pair of $\mathcal{Z}$-related \cinf-structures  $\langle X_1,\ldots, X_{n-r}\rangle$ and $\langle \omega_1,\ldots, \omega_{n-r}\rangle$  for $\mathcal{Z},$ an integrating factor  of $\omega_{i_0}$ w.r.t. $\mathcal{S}(\Lambda_{i_0}),$ for $1\leq i_0\leq n-r-1,$  will be called a {\bf relative integrating factor}, in order to emphasize that such factor {\bf does not} convert the corresponding 1-form $\omega_{i_0}$ into an exact differential form.

 \item Given a pair of $\mathcal{Z}$-related \cinf-structures  $\langle X_1,\ldots, X_{n-r}\rangle$ and $\langle \omega_1,\ldots, \omega_{n-r}\rangle$  for $\mathcal{Z},$   a symmetrizing factor for $X_{i_0}$ w.r.t.  $\mathcal{S}(\mathcal{X}_{{i_0}-1}),$ for $2\leq i_0\leq n-r,$   will be called  a {\bf relative symmetrizing factor}, in order to emphasize that such factor {\bf does not} convert the corresponding vector field $X_{i_0}$ into a symmetry of $\mathcal{Z}$. Only when  $i_0=1,$ a  symmetrizing factor for $X_{1}$ w.r.t.  $\mathcal{S}(\mathcal{X}_{0})=\mathcal{Z}$ converts  $X_1$ into a symmetry of $\mathcal{Z}.$
 
\end{enumerate} 

\end{remark}

The following result generalizes, for distributions of arbitrary corank, the relationship between integrating factors and symmetrizing factors  established in Lemma \ref{lema_inverserelation} for distributions of corank 1:
\begin{lemma}\label{relacioninversa}
Let $\mathcal{Z}=\mathcal{S}(\{Z_1,\ldots,Z_r\})$ be an involutive distribution on an open set $U\subset\mathbb{R}^n$. Assume that $\langle X_1,\ldots, X_{n-r}\rangle$ and $\langle \omega_1,\ldots, \omega_{n-r}\rangle$ is a pair of $\mathcal{Z}$-related \cinf-structures. Suppose that there exists ${i_0}\in\{1,\ldots,n-r\}$  such that  $X_{i_0}\,\lrcorner\, \omega_{i_0}=1$. A non-vanishing smooth function $f$ is a symmetrizing factor for $X_{i_0}$ w.r.t.  $\mathcal{S}(\mathcal{X}_{i_0 -1})$ if and only if the function  $\mu=1/f$ is a $\mathcal{S}(\Lambda_{i_0})$-integrating factor of $\omega_{i_0}$.
\end{lemma}

\begin{proof}
By Definition \ref{Qsolvable},  the vector field  $X_{i_0}$ is a $\mathcal{C}^{\infty}$-symmetry of $\mathcal{S}(\mathcal{X}_{i_0 -1})$ and therefore we can write, for $1\leq i\leq r$ and $1\leq j\leq i_0-1:$
\begin{equation} \label{cuenta0}
[X_{i_0},Z_i]=\lambda_i X_{i_0}+A_i, \qquad     [X_{i_0},X_j]=\gamma_j X_{i_0}+B_j
\end{equation}
for certain smooth functions $\lambda_i,\gamma_j$ and vector fields $A_i,B_j\in \mathcal{S}(\mathcal{X}_{i_0 -1}).$ 
According to Corollary \ref{coro_symfactor} 
a  non-vanishing smooth function $f_{i_0}$ is a symmetrizing factor for $X_{i_0}$ w.r.t.  $\mathcal{S}(\mathcal{X}_{i_0 -1})$ if and only if $f_{i_0}$ satisfies the system
\begin{equation}\label{symfactor}
    \lambda_i-\dfrac{Z_i(f)}{f}=0, \quad
    \gamma_j-\dfrac{X_j(f)}{f}=0, \quad \mbox{for}\quad 1\leq i\leq r,\,1\leq j\leq i_0-1.
\end{equation}

On the other hand, by  Definition \ref{Qsolvabledual}, $\mathcal{I}(\Lambda_{i_0-1})$ is a differential ideal 
and hence there exist 1-forms $\sigma, \sigma_{k}$ such that
\begin{equation}\label{dw}
    d\omega_{i_0}=\sigma\wedge \omega_{i_0}+\sum_{k=i_0+1}^{n-r} \sigma_{k}\wedge \omega_k.
\end{equation} 
Therefore for any non-vanishing smooth function $\mu$   we have
\begin{equation}
       d(\mu\omega_{i_0})= d\mu\wedge\omega_{i_0}+\mu d\omega_{i_0}=(d\mu+\mu\sigma)\wedge\omega_{i_0}+ \displaystyle\sum_{k=i_0+1}^{n-r} \mu \sigma_{k}\wedge \omega_k.
\end{equation}
It follows that 
 the function $\mu$ satisfies $d(\mu\omega_{i_0})\in \mathcal{I}(\Lambda_{i_0})$ if and only if $d\mu+\mu\sigma\in\mathcal{S}(\Lambda_{i_0-1})=\mathcal{S}(\mathcal{X}_{i_0 -1})^{\circ},$ i.e., if and only if $$Z_i\,\lrcorner\, (d\mu+\mu\sigma) =X_j\,\lrcorner\, (d\mu+\mu\sigma)=0, \quad \mbox{for}\quad 1\leq i\leq r,\,1\leq j\leq i_0-1.$$
These conditions can be written in the form
\begin{equation}\label{intfactor}
     \dfrac{Z_i(\mu)}{\mu}+    Z_i\,\lrcorner\,\sigma = 0,\quad
      \dfrac{X_j(\mu)}{\mu}+    X_j\,\lrcorner\,\sigma = 0 \quad \mbox{for}\quad 1\leq i\leq r,\,1\leq j\leq i_0-1.
\end{equation}

In order to link the solutions of systems \eqref{symfactor} and \eqref{intfactor} we need to calculate $Z_i\,\lrcorner\,\sigma$ and $X_j\,\lrcorner\,\sigma.$

First observe that since $\omega_{i_0}\in \mathcal{S}(\Lambda_{i_0-1})=\mathcal{S}(\mathcal{X}_{i_0 -1})^{\circ}$ and $A_i,B_j\in \mathcal{S}(\mathcal{X}_{i_0 -1})$  equations \eqref{cuenta0} imply that
\begin{equation} \label{cuenta1}
[X_{i_0},Z_i]\,\lrcorner\,\omega_{i_0}=\lambda_i 
, \qquad   [X_{i_0},X_j]\,\lrcorner\,\omega_{i_0}=\gamma_j 
.
\end{equation}
Since $X_{i_0}\,\lrcorner\,\omega_{i_0}=1$ and,  for  $1\leq i\leq r,$  we have that $Z_i\,\lrcorner\,\omega_{i_0}=0,$   by \eqref{cuenta1} we can write
\begin{equation}\label{sobreZi}
    d\omega_{i_0}(Z_i,X_{i_0})=Z_i(X_{i_0}\,\lrcorner\,\omega_{i_0})-X_{i_0}(Z_i\,\lrcorner\,\omega_{i_0})-[Z_i,X_{i_0}]\,\lrcorner\,\omega_{i_0}=\lambda_i.
\end{equation} 
Similarly, for   $1\leq j\leq i_0-1,$
\begin{equation}\label{sobreXj}
   d\omega_{i_0}(X_j,X_{i_0})=X_j(X_{i_0}\,\lrcorner\,\omega_{i_0})-X_{i_0}(X_j\,\lrcorner\,\omega_{i_0})-[X_j,X_{i_0}]\,\lrcorner\,\omega_{i_0}=\gamma_j.
\end{equation}
On the other hand, since $Z_i,X_j,X_{i_0}\in \mathcal{S}(\mathcal{X}_{i_0})=\mathcal{S}(\Lambda_{i_0})^{\perp},$ it follows from \eqref{dw} that
\begin{equation}\label{cuenta3}
  \begin{array}{lll}
  d\omega_{i_0}(Z_i,X_{i_0})&=&\sigma\wedge \omega_{i_0} (Z_i,X_{i_0})=Z_i\,\lrcorner\,\sigma,\\
    d\omega_{i_0}(X_j,X_{i_0})&=&\sigma\wedge \omega_{i_0} (X_j,X_{i_0})=X_j\,\lrcorner\,\sigma.
  \end{array}
\end{equation} 
 Equations \eqref{sobreZi}, \eqref{sobreXj} and \eqref{cuenta3} show that 
\begin{equation}\label{fin1}
    Z_i\,\lrcorner\,\sigma= \lambda_i, \quad 
      X_j\,\lrcorner\,\sigma=  \gamma_j,  \quad \mbox{for}\quad 1\leq i\leq r,\,1\leq j\leq i_0-1.
\end{equation}
In consequence, the conditions given in \eqref{intfactor} can be written in the form
\begin{equation}\label{intfactor2}
    \begin{array}{ll}
        \dfrac{Z_i(\mu)}{\mu}+\lambda_i = 0,\quad
      \dfrac{X_j(\mu)}{\mu}+  \gamma_j= 0 \quad \mbox{for}\quad 1\leq i\leq r,\,1\leq j\leq i_0-1.
      \end{array}
\end{equation}
It is clear that $f$ satisfies \eqref{symfactor} if and only if $\mu=1/f$ satisfies system \eqref{intfactor2}, which concludes the proof.
\end{proof}

\begin{remark}\label{notasT43}
\begin{enumerate}
    \item  When $r=n-1$ Lemma \ref{relacioninversa} corresponds to Lemma \ref{lema_inverserelation}: in this case $i_0=1$ and $d(\mu \omega_1)=0,$ i.e., $\mu=1/f$ is an integrating factor of $\omega_{1}$ if and only if $f$ is a symmetrizing factor of $X_1$ respect to $\mathcal{Z}$.
    \item  If $\langle X_1,\ldots, X_{n-r}\rangle$ and $\langle \omega_1,\ldots, \omega_{n-r}\rangle$ is a pair of dually $\mathcal{Z}$-related \cinf-structures, then the requirement $X_{i_0}\,\lrcorner\, \omega_{i_0}=1$ stated in Lemma \ref{relacioninversa} is satisfied for any  ${i_0}\in\{1,\ldots,n-r\}.$    
\end{enumerate}
\end{remark}

Next we investigate the relationship between $\mathcal{S}(\Lambda_{i_0})$-intregrating factors and symmetrizing factors w.r.t $\mathcal{S}(\mathcal{X}_{i_0-1})$ without the  requirement $X_{i_0}\,\lrcorner\,\omega_{i_0}=1$ assumed in Lemma \ref{relacioninversa}.
According to \cite[Remark 3.4 (a)]{pancinf-structures} 
the function $h_{i_0}:=1/(X_{i_0}\,\lrcorner\,\omega_{i_0})$ is well defined and $\langle X_1,\ldots,h_{i_0}X_{i_0},\ldots,X_{n-r}\rangle$ and $\langle \omega_1,\ldots,\omega_{n-r}\rangle$ are $\mathcal{Z}$-related. Since these  \cinf-structures are  in the conditions of Lemma \ref{relacioninversa}, a function $\widetilde{f}$ is a symmetrizing factor of $h_{i_0}X_{i_0}$ w.r.t. $\mathcal{S}(\mathcal{X}_{i_0-1})$ if and only if the function $\mu=1/\widetilde{f}$ is a $\mathcal{S}(\Lambda_{i_0})$-integrating factor for $\omega_{i_0}$. On the other hand, it is clear that a function $f$ is a symmetrizing factor for $X_{i_0}$ w.r.t. $\mathcal{S}(\mathcal{X}_{i_0-1})$ if and only if $\widetilde{f}:=f/h_{i_0}$ is a symmetrizing factor for $h_{i_0} X_{i_0}$ w.r.t. $\mathcal{S}(\mathcal{X}_{i_0-1})$. Therefore $f$ is a symmetrizing factor for $X_{i_0}$ w.r.t. $\mathcal{S}(\mathcal{X}_{i_0-1})$ if and only if 
$$
\mu=1/\widetilde{f}=h_{i_0}/f=1/(fX_{i_0}\,\lrcorner\,\omega_{i_0})
$$
is a $\mathcal{S}(\Lambda_{i_0})$-integrating factor of $\omega_{i_0}$.

This discussion generalizes the result of Theorem \ref{factorgeneral}, established for 1-corank distributions, for \cinf-structures of distributions of arbitrary corank: 
 
\begin{theorem}\label{relacioninversa2}
Let $\mathcal{Z}=\mathcal{S}(\{Z_1,\ldots,Z_r\})$ be an involutive distribution on an open set $U\subset\mathbb{R}^n,$ $r<n.$ Assume that $\langle X_1,\ldots, X_{n-r}\rangle$ and $\langle \omega_1,\ldots, \omega_{n-r}\rangle$ are $\mathcal{Z}$-related \cinf-structures. 
For any ${i_0}\in\{1,\ldots,n-r\},$ a non-vanishing smooth function $f_{i_0}$ is a symmetrizing factor for $X_{i_0}$ w.r.t. $\mathcal{S}(\mathcal{X}_{i_0 -1})$ if and only if the function $\mu_{i_0}=1/(f_{i_0}X_{i_0}\,\lrcorner\, \omega_{i_0})$ is a $\mathcal{S}(\Lambda_{i_0})$-integrating factor for $\omega_{i_0}$.
\end{theorem}

\begin{remark}\label{rem_ventajaSim}
In the search of integral manifolds of an involutive distribution by using a \cinf-structure, a very convenient case appears when some of the $X_i$ in the $\mathcal{C}^{\infty}$-structure is in fact a symmetry of the distribution $\mathcal{S}(\mathcal{X}_{i-1}).$ In this case, according to Theorem \ref{relacioninversa2}, the function $1/X_i\contract \omega_i$ is a $\mathcal{S}(\Lambda_{i})$-integrating factor. The corresponding Pfaffian equation obtained by the restriction of $\omega_i$ can be solved by a quadrature (see Remark \ref{relativeIF}). In particular, the integrability by quadratures in terms of solvable structures  \cite{basarab,hartl1994solvable} follows as a corollary of Theorem \ref{relacioninversa2}.
\end{remark}

\section{Examples}\label{sec5}

In this section we present two examples to show the relationships between relative integrating factors and relative symmetrizing factors for \cinf-structures that have been theoretically obtained in the previous section. 

\begin{example}\label{ex51}

In this example we consider again the distribution that has been integrated in  Example \ref{ex_distribucion} by using the procedure based on \cinf-structures described in Section \ref{algorithm}. Once  the Pfaffian equations $\omega_2\equiv 0$ and $\widetilde{\omega}_1\equiv 0,$ respectively defined by \eqref{omega2_ex} and \eqref{omega1_ex1}, have been integrated, we can obtain, as a consequence, the relative integrating factors from the corresponding first integrals \eqref{I2sol} and \eqref{I1_ex1}. Then  we will use Theorem \ref{relacioninversa2} to construct the relative symmetrizing factors. Finally, we will show how these relative symmetrizing factors could be used to construct a solvable structure for the distribution \eqref{Z_ex1}.

In  Example \ref{ex_distribucion} we have obtained, for the distribution $\mathcal{Z}=\mathcal{S}(\{Z_1,Z_2\})$ spanned by the vector fields given in \eqref{Z_ex1}, a \cinf-structure $\langle X_1,X_2\rangle$ formed by the vector fields
\begin{equation}\label{X_ex1}
    X_1=x_4\partial_{x_2}+\partial_{x_3}, \quad X_2=\partial_{x_4}.
\end{equation} 
\begin{enumerate}
    \item {\bf $\mathcal{S}(\Lambda_2)$-integrating factor and $\mathcal{S}(\mathcal{X}_1)$-symmetrizing factor:}
    
    In the second step of the integration procedure associated to the \cinf-structure, we have calculated the  function \eqref{I2sol}:
\begin{equation}\label{I2_ex2}
I_2=   x_1+\dfrac{x_4}{x_2-x_3 x_4},\end{equation} which is a first integral for the involutive distribution $\mathcal{X}_1=\mathcal{S}(\{Z_1,Z_2,X_1\}).$ In particular, $I_2$ is a true first integral for $\mathcal{Z}.$ Recall that such first integral has been obtained by integrating the Pfaffian equation $\omega_2\equiv 0,$ and therefore, there exists a function $\mu_2=\mu_2(x_1,x_2,x_3,x_4)$ such that 
$$\mu_2\omega_2=dI_2.$$ It can be checked that this function $\mu_2,$ which is in fact an integrating factor for the Frobenius integrable 1-form $\omega_2,$ becomes
\begin{equation}
    \mu_2=-(x_2-x_3x_4)^{-2}.
\end{equation}
Therefore, according to Theorem \ref{relacioninversa2}, the function 
\begin{equation}\label{f2_ex1}f_2=\dfrac{1}{\mu_2 X_{2}\lrcorner\,\omega_2}=\dfrac{1}{\mu_2 \Delta}=\dfrac{(x_2-x_3x_4)^2}{x_2}\end{equation}
is a symmetrizing factor for $X_2$ w.r.t. $\mathcal{S}(\mathcal{X}_1)=\mathcal{S}(\{Z_1,Z_2,X_1\}).$ In other words, the vector field 
\begin{equation}\label{Y2_ex1}
    Y_2:=f_2X_2=\dfrac{(x_2-x_3x_4)^2}{x_2}\partial_{x_4}
\end{equation} is a symmetry for the involutive distribution  $\mathcal{S}(\mathcal{X}_1)=\mathcal{S}(\{Z_1,Z_2,X_1\}).$ 

Observe, however, that $Y_2$ {\bf is not a symmetry} for $\mathcal{Z}=\mathcal{S}(\{Z_1,Z_2\}),$ because 
$$[Y_2,Z_1]=
-\left(\dfrac{x_3(x_2-x_3x_4)}{x_2}\right)^2\left( Z_1+x_3X_1\right)\notin \mathcal{Z}.$$
 \item {\bf $\mathcal{S}(\Lambda_1)$-integrating factor and $\mathcal{S}(\mathcal{X}_0)$-symmetrizing factor: }

In the fourth step of the integration procedure performed in Example \ref{ex_distribucion}, we have obtained the function \eqref{I1_ex1} 
$$I_1=\dfrac{1+x_3(C_2-x_1)}{x_2x_3},$$ by integrating the 1-form $\widetilde{\omega}_1$ given in \eqref{omega1_ex1}, which is Frobenius integrable on $N_1\subseteq \mathbb{R}^3.$ As before, a corresponding integrating factor $\widetilde{\mu}_1=\widetilde{\mu}_1(x_1,x_2,x_3;C_2)$ for $\widetilde{\omega}_1$ arises from the condition $dI_1=\widetilde{\mu}_1\widetilde{\omega}_1,$ giving rise to the function 
\begin{equation}\label{mu1_ex1_tilde}
\widetilde{\mu}_1(x_1,x_2,x_3;C_2)=-\dfrac{1+(C_2-x_1)x_3}{x_2^2 x_3^2}.
\end{equation}
In order to obtain a $\mathcal{S}(\Lambda_1)$-integrating factor for the 1-form $\omega_1$ given in \eqref{omega1_ex}, we consider the function  $\mu_1={\mu}_1(x_1,x_2,x_3,x_4)$ obtained from \eqref{mu1_ex1_tilde} by replacing  $C_2$ by the right-hand side  of \eqref{I2_ex2}, i.e., 
\begin{equation}\label{mu1_ex1}
    \mu_1=\widetilde{\mu}_1(x_1,x_2,x_3;I_2(x_1,x_2,x_3,x_4))=-\dfrac{1}{x_2x_3^2(x_2-x_3 x_4)}.
\end{equation}


It can be checked that the 1-form
$$\mu_1\omega_1=-\dfrac{1}{x_2}dx_1-\dfrac{1}{x_2x_3(x_2-x_3x_4)}dx_2-\dfrac{1}{x_2x_3^2}dx_3$$ satisfies
$$d(\mu_1\omega_1)=\left(-\dfrac{1}{x_2^2(x_2-x_3x_4)^2}dx_2\right)\wedge\omega_2\in \mathcal{I}(\Lambda_1),$$  which shows that, although \eqref{mu1_ex1} {\bf is not an integrating factor} for $\omega_1,$ the 1-form $\mu_1\omega_1$ is closed module $\Lambda_1,$ i.e., the function $\mu_1$ in \eqref{mu1_ex1} is an integrating  factor for $\omega_1$ w.r.t. $\mathcal{S}(\Lambda_1)=\mathcal{S}(\{\omega_2\}).$

Now we use the relative integrating factor \eqref{mu1_ex1} to construct, by using  Theorem \ref{relacioninversa2}, a  symmetrizing factor for $X_1$ w.r.t. $\mathcal{S}(\mathcal{X}_0)=\mathcal{S}(\{Z_1,Z_2\})=\mathcal{Z}.$  In this case the  function 
\begin{equation}\label{f1_ex1}f_1=\dfrac{1}{\mu_1 X_{1}\lrcorner\,\omega_1}=-\dfrac{1}{\mu_1 \Delta}=-{x_3^2(x_2-x_3 x_4)}\end{equation}
is a symmetrizing factor for $X_1$ w.r.t. $\mathcal{Z}.$ In other words, the vector field 
\begin{equation}\label{Y1_ex1}
    Y_1:=f_1X_1=-x_4x_3^2(x_2-x_3 x_4)\partial_{x_2}-x_3^2(x_2-x_3 x_4)\partial_{x_3}
\end{equation} is a  symmetry for $\mathcal{Z}=\mathcal{S}(\{Z_1,Z_2\}).$  In fact, it can be checked that 
$$[Y_1,Z_1]=0,\quad [Y_1,Z_2]=f_1Z_1.$$
\item {\bf A solvable structure for $\mathcal{Z}:$} 

As a consequence of the previous calculations, we observe that the vector fields $Y_1$ and $Y_2,$ given in \eqref{Y1_ex1} and  \eqref{Y2_ex1}, respectively  
\begin{equation}
    \label{solvable_ex1}
        Y_1=-x_4x_3^2(x_2-x_3 x_4)\partial_{x_2}-x_3^2(x_2-x_3 x_4)\partial_{x_3},\quad  Y_2=\dfrac{(x_2-x_3x_4)^2}{x_2}\partial_{x_4}
\end{equation} provide a solvable structure $\langle Y_1,Y_2\rangle$ for $\mathcal{Z},$ because the following commutations relationships hold: 
$$\left[Y_1,Z_1\right]=0,\quad  \left[Y_1,Z_2\right]=-x_3^2(x_2-x_3x_4)Z_1$$
and 
$$\begin{array}{l}
\left[Y_2,Z_1\right]=-\dfrac{x_3(x_2-x_3 x_4)^2}{x_2^2} Z_1-\dfrac{(x_2-x_3 x_4)}{x_2^2 }  Y_1,\\
    \left[Y_2,Z_2\right]=0,\\
    \left[Y_2,Y_1\right]=-\dfrac{x_3^2(x_2-x_3x_4)^3}{x_2^2} Z_1-\dfrac{x_3^2x_4(x_2-x_3x_4)}{x_2^2} Y_1.
\end{array}
$$

It should be emphasized that the solvable structure $\langle Y_1,Y_2\rangle$ is not necessary to integrate the distribution $\mathcal{Z},$  because this has already been  done in Example \ref{ex_distribucion} by using the \cinf-structure integration method. Such solvable structure  has been constructed to explain better the following considerations  that emerge from a comparative analysis of the integration procedures based on \cinf-structures and solvable structures:

\begin{itemize}
    \item The vector fields \eqref{solvable_ex1} that form the solvable structure have been determined from the vector fields  \eqref{X_ex1}  that define the \cinf-structure and the corresponding relative symmetrizing factors \eqref{f2_ex1} and \eqref{f1_ex1}. These elements are available only when the integration procedure associated to the \cinf-structure, performed in Example \ref{ex_distribucion}, has been completed.  
    \item The Pfaffian equations which appear in the integration procedure associated to a solvable structure are integrable by quadratures. However, a solvable structure is more difficult to calculate than a \cinf-structure, because their elements must satisfy stronger conditions  than in the case of  \cinf-structures. In this example it is clear  that the infinitesimals of the vector fields   \eqref{X_ex1} that form a \cinf-structure are  
    simpler  than the infinitesimals of the vector fields \eqref{solvable_ex1} that define a solvable structure.
\end{itemize}
\end{enumerate}



\end{example}

In the next example we calculate and use a \cinf-structure to solve a third-order ODE which, to the best of our knowledge, cannot be integrated by classical methods. As in the previous example, relative symmetrizing factors will also be used to construct a solvable structure. The infinitesimals of the elements of such solvable structure depend on special functions (of Airy type), which makes very difficult their determination without the previous knowledge of the \cinf-structure. As before, the solvable structure is not necessary to solve the ODE, which is completely integrated by using exclusively the \cinf-structure method.

\begin{example}\label{ex_airy}
Let us consider the third-order ordinary differential equation
\begin{equation}\label{eq_airy}
    u_3 = \displaystyle\frac{u_1(xu_1-x-1)-u_2 (u_1+x)-x^2}{u_1},
\end{equation}
and denote by $$A=\partial_x+u_1\partial_u+u_2\partial_{u_1}+\left(\dfrac{u_1(xu_1-x-1)-u_2 (u_1+x)-x^2}{u_1}\right)\partial_{u_2}$$ its associated vector field, defined on the open set $U=\{(x,u,u_1,u_2)\in J^2(\mathbb{R},\mathbb{R}): u_1\neq 0\}.$  It can be checked that equation (\ref{eq_airy}) only admits  the Lie point symmetry $\partial_u$. Moreover, the associated reduced equation of second-order, does not admit Lie point symmetries. 

We aim    to integrate the trivially involutive distribution $\mathcal{Z}=\mathcal{S}(\{A\})$ by using a \cinf-structure.  Clearly the second-order prolongation \cite{olver86} of the Lie point symmetry $\partial_u$ becomes $X_1=(\partial_u)^{(2)}=\partial_u$ and $X_1$ is a symmetry of $\mathcal{S}(\{A\}).$ Therefore, $X_1=\partial_u$ can  be used as the first element of a \cinf-structure for $\mathcal{S}(\{A\}).$ For the involutive distribution $\mathcal{X}_1=\mathcal{S}(\{A,X_1\}),$ it is not difficult to find the \cinf-symmetry $X_2
$ given  by the vector field  $$ X_2 = \partial_{u_1} +\dfrac{u_2+x}{u_1}\partial_{u_2},$$ which satisfies the following commutation relations: 
\begin{equation}
[X_1,A] =0, \quad
[X_2,A] =X_1+\dfrac{u_2+x}{u_1} X_2, \quad
[X_2,X_1] = 0.
\end{equation}
As the third element of a \cinf-structure we may choose any vector field $X_3$ such that the set $\{A,X_1,X_2,X_3\}$ is linearly independent. By choosing $X_3= \partial {u_2}$ the corresponding determinant \eqref{delta} becomes $$\Delta=X_3\,\lrcorner\,X_2\,\lrcorner\,X_1\lrcorner\,A\,\lrcorner\,\boldsymbol{\Omega}=1.$$ In consequence, the ordered set $\langle X_1,X_2,X_3 \rangle$ defined by the vector fields 
\begin{equation}\label{vectors}
     X_1 = \partial_u,\quad
     X_2 = \partial_{u_1} +\dfrac{u_2+x}{u_1}\partial_{u_2},  \quad
     X_3 = \partial {u_2},
\end{equation} constitute a \cinf-structure  for $\mathcal{S}(\{A\})$.

Once a \cinf-structure has been found, the equation \eqref{eq_airy} can be integrated by following the procedure described at the beginning of Section \ref{algorithm}. In what follows, we just present the results obtained in each step, by means of calculations similar to those that have been described in detail for Example \ref{ex_distribucion}: 
\begin{enumerate}
    \item From the vector fields \eqref{vectors}, a  $\mathcal{Z}$-related \cinf-structure of 1-forms for $\mathcal{S}(\{A\})$ can be constructed as in \eqref{contractedforms}:
\begin{equation}\label{w_ex3}
    \begin{array}{l}
\omega_1 = -u_1 dx + du,\\ \\
\omega_2 = u_2 dx - du_1,\\ \\
\omega_3 = \left(\dfrac{(u_2+x)^2}{u_1}+u_2+x+1-xu_1 \right) dx-\dfrac{u_2+x}{u_1} du_1+d u_2.
\end{array}
\end{equation}
\item The integration of the  completely integrable Pfaffian equation $\omega_3\equiv 0$ yields a corresponding first integral $I_3=I_3(x,u,u_1,u_2)$ that can be expressed in terms of two linearly independent solutions $\phi_1=\phi_1(x)$ and $\phi_2=\phi_2(x)$  of the second-order linear equation
\begin{equation}\label{lineal}
  \phi''(x)-\left(x+\dfrac{1}{4}\right)\phi(x)=0.
\end{equation} Since equation \eqref{lineal} can be transformed, by means of  the translation of the independent variable $x=t-1/4,$ into the Airy equation $\phi''(t)-t\phi(t)=0,$  then  a fundamental set $\{\phi_1,\phi_2\}$  of solutions   to equation \eqref{lineal} can be expressed in terms of the  Airy functions as follows: \begin{equation}\label{airy}
    \phi_1(x)=\mbox{Ai}\left(x+\dfrac{1}{4}\right),\qquad \phi_2(x)=\mbox{Bi}\left(x+\dfrac{1}{4}\right).
\end{equation} It can be checked that the integration of $\omega_3\equiv 0$ leads to the first integral  
\begin{equation}\label{I3_airy}
    I_3 = -\displaystyle\frac{2u_1\phi_2'(x) - (2u_2+u_1+2x)\phi_2(x)}{2u_1\phi_1'(x) - (2u_2+u_1+2x)\phi_1(x)}.
\end{equation} This function is a first integral for the distribution $\mathcal{S}(\{A,X_1,X_2\}).$ As usual, let $\Sigma_{(C_3)}$ denote the corresponding integral manifold implicitly defined by $I_3=C_3,$ $C_3\in \mathbb{R}.$
\item The pullbacks of $\omega_1$ and $\omega_2$ in \eqref{w_ex3},  by the following local parametrization of $\Sigma_{(C_3)}:$ 
$$\iota_3(x,u,u_1)=\left(x,u,u_1, \dfrac{C_3\phi'_1(x)+\phi'_2(x)}{C_3\phi_1(x)+\phi_2(x)}\,u_1-\dfrac{u_1}{2}-x\right),$$
are respectively given by:  
\begin{equation}\label{wr_ex3}
 \begin{array}{l}
\widetilde{\omega}_1:=(\iota_3)^*(\omega_1) = -u_1 dx + du,\\ \\
\widetilde{\omega}_2:=(\iota_3)^*(\omega_2) = -\left(\dfrac{u_1}{2}+x-\dfrac{C_3\phi'_1(x)+\phi'_2(x)}{C_3\phi_1(x)+\phi_2(x)}\,u_1\right)dx - du_1.
\end{array}   
\end{equation}
\item The Pfaffian equation $\widetilde{\omega}_2\equiv 0$ is completely integrable and it can be checked that a  corresponding first integral $I_2=I_2(x,u,u_1;C_3)$ becomes
\begin{equation}\label{I2}
    I_2=G(x;C_3)+\dfrac{u_1 e^{x/2}}{C_3\phi_1(x)+\phi_2(x)},
\end{equation} where  $G=G(x;C_3)$ is a smooth function satisfying 
\begin{equation}\label{G}
G'(x;C_3)=\displaystyle\frac{xe^{x/2}}{C_3 \phi_1(x) +\phi_2(x)}.
\end{equation}
\item A local parametrization for the integral manifold $\Sigma_{(C_2,C_3)}$ of  $\widetilde{\omega}_2\equiv 0,$ implicitly defined by $I_2(x,u,u_1;C_3)=C_2,C_2\in \mathbb{R},$ becomes
$$\iota_2(x,u)=\left(x,u,e^{-x/2}\,\left( C_3 \phi_1(x) +\phi_2(x) \right)\,\left( C_2-G(x;C_3) \right)\right).$$ The pullback of the 1-form $\widetilde{\omega}_1$ defined in \eqref{wr_ex3}  by $\iota_2$ becomes
\begin{equation}\label{wr2_ex3}
   \widehat{\omega}_1:=(\iota_2)^*(\widetilde{\omega}_1)=-e^{-x/2} \left( C_3 \phi_1(x) +\phi_2(x) \right)\left( C_2-G(x;C_3) \right)dx+du.
\end{equation}
\item Since, for this example,  the first element $X_1$ of the \cinf-structure is not only a \cinf-symmetry of $\mathcal{Z}=\mathcal{S}(\{A\})$ but a symmetry,  then, according to Remark \ref{rem_ventajaSim}, the function 
$$\dfrac{1}{X_1\,\lrcorner\,\omega_1}=\dfrac{1}{\Delta}=1$$ is an integrating factor of the 1-form given in \eqref{wr2_ex3}.  In other words, the  Pfaffian equation $\widehat{\omega}_1\equiv 0$ is not only completely integrable, but integrable by a quadrature. A corresponding first integral $I_1=I_1(x,u;C_2,C_3)$ can be expressed in the form
\begin{equation}\label{I1_ex3}
    I_1(x,u;C_2,C_3)=H(x;C_2,C_3)+u,
\end{equation}
where the function $H=H(x;C_2,C_3)$ satisfies 
$$H'(x;C_2,C_3)=-e^{-x/2}\left( C_2-G(x;C_3) \right) \left( C_3 \phi_1(x) +\phi_2(x) \right).$$
\item In this example, we do not need to calculate the integral manifolds of the distribution $\mathcal{S}(\{A\}),$ which corresponds to the second-order prolongations of the solutions of the ODE \eqref{eq_airy}. The general solution of the equation \eqref{eq_airy} can be directly expressed, in implicit form, in terms of the function given in  \eqref{I1_ex3}, as  $H(x;C_2,C_3)+u=C_1,$ where $C_1\in \mathbb{R}.$ This provides the following explicit expression for the general solution of the equation \eqref{eq_airy}, written in terms of the Airy functions given in \eqref{airy} and the function $G(x;C_3)$ that satisfies \eqref{G}:
\begin{equation}\label{solucion_ex3}
    u(x)=\int e^{-x/2}\left( C_2-G(x;C_3) \right) \left( C_3 \phi_1(x) +\phi_2(x) \right)dx+C_1.
\end{equation}
\end{enumerate} 
Let us observe that, as in Example \ref{ex51}, once the integration procedure has been completed, we can now calculate the relative integrating factors and the relative symmetrizing factors for the \cinf-structure defined by the vector fields given in \eqref{vectors}. We emphasize that what follows is not necessary for the integration of the equation, whose general solution has already been obtained in \eqref{solucion_ex3}:

\begin{enumerate}
    \item {\bf $\mathcal{S}(\Lambda_3)$-integrating factor and $\mathcal{S}(\mathcal{X}_2)$-symmetrizing factor: } A function $\mu_3=\mu_3(x,u_1,u_2)$ such that $dI_3=\mu_3\,\omega_3$ becomes
   $$\mu_3 =\dfrac{4u_1 \left( \phi_1'(x) \phi_2(x)-\phi_1(x) \phi_2'(x)  \right)}{\left( 2 u_1 \phi_1'(x) - (2u_2+u_1+2x) \phi_1(x) \right)^2}.$$ Since $X_3\,\lrcorner\,\omega_3=1,$ by Theorem \ref{relacioninversa2}, we have that $f_3=1/\mu_3$ is a $\mathcal{S}(\mathcal{X}_2)$-symmetrizing factor for $X_3,$ i.e., the vector field 
  \begin{equation}\label{f3X3}
    Y_3:=   f_3 X_3 = \displaystyle\frac{\left( 2 u_1 \phi_1'(x) - (2u_2+u_1+2x) \phi_1(x) \right)^2}{4u_1 \left( \phi_1'(x) \phi_2(x)-\phi_1(x) \phi_2'(x)  \right)} \partial_{u_2},
  \end{equation}  is a symmetry of the distribution $\mathcal{X}_2=\mathcal{S}(\{A,X_1,X_2\}).$ However, $Y_3$ {\bf is not a generalized symmetry} for the equation \eqref{eq_airy}.
    \item {\bf $\mathcal{S}(\Lambda_2)$-integrating factor and $\mathcal{S}(\mathcal{X}_1)$-symmetrizing factor: } A function $\widetilde{\mu}_2=\widetilde{\mu}_2(x,u_1;C_3)$ such that $dI_2=\widetilde{\mu}_2\,\widetilde{\omega}_2$ becomes
   $$\widetilde{\mu}_2(x,u_1;C_3) =-\dfrac{e^{x/2}}{C_3 \phi_1(x) +\phi_2(x)}.$$ In order to obtain a $\mathcal{S}(\Lambda_2)$-integrating factor  for $\omega_2,$ we replace $C_3$ by the right-hand side of \eqref{I3_airy}:
   $$\mu_2(x,u,u_1,u_2)= -\dfrac{ 2 u_1\phi_1'(x)- (2 u_2+u_1+2x )\phi_1(x)}{2 u_1 e^{-x/2}\left( \phi_1'(x)\phi_2(x)-\phi_1(x) \phi_2'(x)\right)}.$$ Since $X_2\,\lrcorner\,\omega_2=-1,$ by Theorem \ref{relacioninversa2}, we have that $f_2=-1/{\mu}_2$ is a $\mathcal{S}(\mathcal{X}_1)$-symmetrizing factor for $X_2,$ i.e., the vector field 
  \begin{equation}\label{f2X2}
     Y_2:=   f_2 X_2 = \dfrac{2 u_1 e^{-x/2}\left( \phi_1'(x)\phi_2(x)-\phi_1(x) \phi_2'(x)\right)}{ 2 u_1\phi_1'(x)- (2 u_2+u_1+2x )\phi_1(x)} \left(\partial_{u_1} +\dfrac{u_2+x}{u_1}\partial_{u_2}\right),  
  \end{equation}  is a symmetry of the distribution $\mathcal{X}_1=\mathcal{S}(\{A,X_1\}),$ but {\bf is not a generalized symmetry} for equation \eqref{eq_airy}.
    \item {\bf $\mathcal{S}(\Lambda_1)$-integrating factor and $\mathcal{S}(\mathcal{X}_0)$-symmetrizing factor: } Since, in this example, the first element $X_1$ of the \cinf-structure is already a symmetry of $\mathcal{S}(\{A\}),$  the $\mathcal{S}(\Lambda_1)$-integrating factor is  $\mu_1=1$ and  $f_1=1$ is the symmetrizing factor of $X_1.$ 
\end{enumerate}
As a consequence of the previous calculations, we have obtained the vector fields: 
\begin{equation}\label{solvable_ex3}
\begin{array}{l}
    Y_1:= f_1 X_1 = \partial_u,\\
     Y_2:=f_2 X_2 = \dfrac{2 u_1 e^{-x/2}\left( \phi_1'(x)\phi_2(x)-\phi_1(x) \phi_2'(x)\right)}{ 2 u_1\phi_1'(x)- (2 u_2+u_1+2x )\phi_1(x)} \left(\partial_{u_1} +\dfrac{u_2+x}{u_1}\partial_{u_2}\right),  \\
   Y_3:=  f_3 X_3 = \displaystyle\frac{\left( 2 u_1 \phi_1'(x) - (2u_2+u_1+2x) \phi_1(x) \right)^2}{4u_1 \left( \phi_1'(x) \phi_2(x)-\phi_1(x) \phi_2'(x)  \right)} \partial_{u_2},
\end{array}
\end{equation}
which constitute a solvable structure $\left\langle Y_1,Y_2,Y_3\right\rangle$ for $\mathcal{S}(\{A\}).$ This solvable structure has been calculated just for illustration purposes, because it appears only when the general solution \eqref{solucion_ex3} of the equation has already been obtained. It should be observed that the infinitesimals of the vector fields that define this solvable structure are not obvious: they depend on Airy functions and are therefore much more difficult to find that the vector fields \eqref{vectors} that define the \cinf-structure.   


\end{example}
\section{Final remarks}
	The integrability of systems of ODEs can be formulated as the geometrical problem of integrating the distribution generated by the vector field associated to  the system.
	More generally, the integrability of  a general involutive distribution $\mathcal{Z}=\mathcal{S}(\{Z_1,\ldots,Z_r\})$ is guaranteed by Frobenius theorem. However this theorem does not provide a method to construct the corresponding integral manifolds. One of the most fruitful methods of integration is provided by solvable structures \cite{basarab,sherring1992geometric,hartl1994solvable,Barco2001,Barco2002,barco2001similarity,BarcoPDE}.
	The first element $X_1$ in a solvable structure must be a symmetry of the distribution $\mathcal{Z}$. The second  element $X_2$  must be a symmetry of the distribution generated by $\{Z_1,\ldots,Z_r, X_1\}$, and so on. Here $X_2$ is not necessarily a symmetry of $\mathcal{Z}$. The interest on solvable structures is that they allow the integration by quadratures of the distribution. The main drawback of the method is that, in general, it is very difficult to find the corresponding vector fields in practice. This has motivated several studies that can be useful for that aim. For instance, for ODEs, prior knowledge of some symmetry generators can help to build a solvable structure, as  observed by  P. Morando and D. Catalano-Ferraioli in \cite{Catalano_Paola,Catalano_nonlocal}. However, many times the  equations under study lack Lie point symmetries, or the generalized symmetries necessary for the construction of the solvable structure are so intricate that they are very difficult to determine in practice.  
 

	In the recent paper  \cite{pancinf-structures}, the concept of solvable structure has  been generalized to  $\mathcal{C}^\infty$-structures,  by requiring  the vector fields $X_1, X_2, \ldots, X_{n-r}$  be $\mathcal{C}^\infty$-symmetries of the previous distributions. 
	These new structures are quite more easy to be found than solvable structures, because  weaker conditions are required for their elements. A procedure for the stepwise integration of $\mathcal{Z}$ via \cinf-structures is included in the paper. 
 Such procedure, as in the case of solvable structures, is based on  successive integrations of Pfaffian equations. For \cinf-structures, such Pfaffian equations are also completely integrable, although  not necessarily integrable by quadratures, as in the case of solvable structures. However, they can still be integrated, even without  prior knowledge of an integrating factor, as illustrated by the examples presented in this paper (see also the examples given in \cite{pancinf-structures}).  Furthermore, since 
  each Pfaffian equation is defined in a space with one less dimension than in the previous stage, its integration (with or without integrating factors) becomes more and more easy.

  A remarkable feature of the \cinf-structures procedure is that once the Pfaffian equation at any intermediate step, defined on a $j$th-dimensional space, has been integrated, then it is possible to construct a corresponding  integrating factor, which depends on $n-j$ constants of integration. As a consequence,  a relative integrating factor (depending on the $n$ variables) can be constructed by replacing  such constants by the corresponding  $n-j$ first integrals  calculated in the previous stages.   This procedure considerably simplifies the calculation of such relative integrating factors. 

  The well-known relationship between symmetries and integrating factors for first-order ODEs has been generalized in this  paper for \cinf-structures and involutive distributions of arbitrary corank in Theorem \ref{relacioninversa2}. This result establishes
a kind of reciprocal relationship between the relative symmetrizing factors of the elements of the \cinf-structure and the relative integrating factors.  

The examples presented in this paper clearly illustrate that the relative integrating factors, and hence the relative symmetrizing factors, can be very intricate. Therefore, if only 
solvable structures were used to integrate the distribution then  such factors would  extremely  complicate  the expressions of the infinitesimals of the vector fields (see, for instance, the expressions \eqref{solvable_ex3} in Example \ref{ex_airy}). Since this is not necessary, because the integration procedure can be carried out by using the \cinf-structure, this makes clear the  advantage of the \cinf-structure method.

\section*{Acknowledgments}
C. Muriel and A. Ruiz acknowledge the financial support from 
{\it   Junta de Andaluc\'ia} (Spain) by means of  the project ProyExcel$\_$00780. The authors also acknowledge the financial support from {\it   Junta de Andaluc\'ia} (Spain)  to the  research group FQM--377. 

\bibliographystyle{unsrt}  
\bibliography{references.bib}

\begin{thebibliography}{10}

\bibitem{ibragimovlibro}
N.~H. Ibragimov.
\newblock {\em A {P}ractical {C}ourse in {D}ifferential {E}quations and
  {M}athematical {M}odelling: {C}lassical and {N}ew {M}ethods, {N}onlinear
  {M}athematical {M}odels, {S}ymmetry and {I}nvariance {P}rinciples}.
\newblock World Scientific, Beijing, 2010.

\bibitem{sherring1992geometric}
J.~Sherring and G.~Prince.
\newblock Geometric aspects of reduction of order.
\newblock {\em Trans. Amer. Math. Soc.}, 334(1):433--453, 1992.

\bibitem{lychagin1991}
S.V. Duzhin and V.~V. Lychagin.
\newblock Symmetries of distributions and quadrature of ordinary differential
  equations.
\newblock {\em Acta Appl. Math.}, 29:29--57, 1991.

\bibitem{Dubrov1998}
B.~Dubrov and B.~Komrakov.
\newblock {\em Symmetries of Completely Integrable Distributions}, pages
  393--405.
\newblock Springer Netherlands, Dordrecht, 1998.

\bibitem{olver86}
P.~J. Olver.
\newblock {\em Applications of {L}ie groups to differential equations}, volume
  107.
\newblock Springer-Verlag, New York, 1986.

\bibitem{stephani}
H.~Stephani.
\newblock {\em Differential Equations: Their Solutions Using Symmetry}.
\newblock Cambridge University Press, New York, 1989.

\bibitem{blumanlibro}
G.~W. Bluman and S.~C. Anco.
\newblock {\em Symmetry and Integration Methods for Differential Equations}.
\newblock Springer-Verlag, New York, 2002.

\bibitem{ovsiannikovlibro}
L.~V. Ovsiannikov.
\newblock {\em Group analysis of differential equations}.
\newblock Academic press, New York, 2014.

\bibitem{blumankumeilibro}
G.~W. Bluman and S.~Kumei.
\newblock {\em Symmetry and Differential Equations}.
\newblock Springer-Verlag, New York, 1989.

\bibitem{basarab}
P.~Basarab-Horwath.
\newblock Integrability by quadratures for systems of involutive vector fields.
\newblock {\em Ukrainian Math. Zh.}, 43:1330--1337, 1991.

\bibitem{muriel01ima1}
C.~Muriel and J.~L. Romero.
\newblock New methods of reduction for ordinary differential equations.
\newblock {\em IMA J. Appl. Math.}, 66(2):111--125, 2001.

\bibitem{gaetatwisted}
G.~Gaeta.
\newblock {Twisted symmetries of differential equations}.
\newblock {\em J. Nonlinear Math. Phys.}, 16:107--136, 2009.

\bibitem{muriel_evolution}
C.~Muriel and J.~L. Romero.
\newblock Evolution of the concept of $\lambda$-symmetries.
\newblock In N.~Euler and M.C. Nucci, editors, {\em Nonlinear Systems and Their
  Remarkable Mathematical Structures}, pages 158--188. Chapman and Hall/CRC,
  2019.

\bibitem{pucci}
E.~Pucci and G.~Saccomandi.
\newblock On the reduction methods for ordinary differential equations.
\newblock {\em J. Phys. A: Math. Gen.}, 35:6145--6155, 2002.

\bibitem{gaetamorando}
G.~Gaeta and P.~Morando.
\newblock On the geometry of $\lambda-$symmetries and {PDE} reduction.
\newblock {\em J. Phys. A: Math. Gen.}, 37(27):6955--6975, 2004.

\bibitem{gaeta2}
G.~Cicogna, G.~Gaeta, and P.~Morando.
\newblock On the relation between standard and $\mu$-symmetries for {PDE}s.
\newblock {\em J. Phys. A: Math Gen.}, 37:9467--9486, 2004.

\bibitem{murielolver}
C.~Muriel, J.~L. Romero, and P.~J. Olver.
\newblock Variational {$C\sp \infty$}-symmetries and {E}uler-{L}agrange
  equations.
\newblock {\em J. {D}iff. {E}q.}, 222(1):164--184, 2006.

\bibitem{cicognagaeta_noether}
G.~Cicogna and G.~Gaeta.
\newblock {N}oether theorem for {$\mu$}-symmetries.
\newblock {\em J. Phys. A: Math. Theor.}, 40(39):11899--11921, 2007.

\bibitem{morando_mu_symmetries}
P.~Morando.
\newblock Deformation of {L}ie derivative and $\mu$-symmetries.
\newblock {\em J. Phys. A: Math. Theor.}, 40:11547--11559, 2007.

\bibitem{cicogna2012generalization}
G.~Cicogna, G.~Gaeta, and S.~Walcher.
\newblock A generalization of $\lambda$-symmetry reduction for systems of odes:
  {$\sigma-$}symmetries.
\newblock {\em J. Phys. A: Math. Theor.}, 45(35), 2012.

\bibitem{cicogna2013dynamical}
G.~Cicogna, G.~Gaeta, and S.~Walcher.
\newblock Dynamical systems and $\sigma$-symmetries.
\newblock {\em J. Phys. A: Math. Theor.}, 46(23):235204, 2013.

\bibitem{gaetacollective}
G.~Gaeta.
\newblock Simple and collective twisted symmetries.
\newblock {\em J. Nonlinear Math. Phys.}, 21(4):593--627, 2014.

\bibitem{gaeta_frobenius}
G.~Gaeta.
\newblock Symmetry and {L}ie-{F}robenius reduction of differential equations.
\newblock {\em J. Phys. A: Math. Theor.}, 48:0152202--015225, 2015.

\bibitem{PaolaPfaffian}
P.~Morando and S.~Sammarco.
\newblock Variational problems with symmetry: A {P}faffian system approach.
\newblock {\em Acta Appl. Math.}, 120:255--274, 2012.

\bibitem{paola.frobenius}
P.~Morando.
\newblock Reduction by {$\lambda-$}symmetries and {$\sigma-$}symmetries: a
  {F}robenius approach.
\newblock {\em J. Nonlinear Math. Phys.}, 22(1):47--59, 2015.

\bibitem{pancinf-structures}
A.~J. Pan-Collantes, A.~Ruiz, C.~Muriel, and J.~L. Romero.
\newblock $\mathcal{C}^{\infty}$-symmetries of distributions and integrability.
\newblock {\em J. {D}iff. {E}q.}, 348:126--153, 2023.

\bibitem{Morita}
S.~Morita.
\newblock {\em Geometry of Differential Forms}.
\newblock American Mathematical Society, Rhode Island, 2001.

\bibitem{warner}
F.~W. Warner.
\newblock {\em Foundations of differentiable manifolds and {L}ie groups},
  volume~94.
\newblock Springer Science \& Business Media, 1983.

\bibitem{bryant2013exterior}
R.~L. Bryant, S.~S. Chern, R.~B. Gardner, H.~L. Goldschmidt, and P.~A.
  Griffiths.
\newblock {\em Exterior differential systems}, volume~18.
\newblock Springer Science \& Business Media, 2013.

\bibitem{hartl1994solvable}
T.~Hartl and C.~Athorne.
\newblock Solvable structures and hidden symmetries.
\newblock {\em J. Phys. A: Math. Gen.}, 27(10):3463, 1994.

\bibitem{ChrisAthorne1998}
C.~Athorne.
\newblock On the {L}ie symmetry algebra of a general ordinary differential
  equation.
\newblock {\em J. Phys. A: Math. Gen.}, 31:6605, 1998.

\bibitem{Barco2001}
M.~A. Barco and G.~Prince.
\newblock Solvable symmetry structures in differential form applications.
\newblock {\em Acta Appl. Math.}, 66:89--121, 2001.

\bibitem{Barco2002}
M.~A. Barco.
\newblock Solvable structures and their application to a class of {C}auchy
  problem.
\newblock {\em Eur. J. Appl. Math.}, 13(5):449--477, 2002.

\bibitem{BarcoPDE}
M.~A. Barco and G.~Prince.
\newblock New symmetry solution techniques for first-order nonlinear {PDE}s.
\newblock {\em Appl. Math. Comp.}, 124:169--196, 2001.

\bibitem{barco2001similarity}
M.~A. Barco.
\newblock An application of solvable structures to classical and nonclassical
  similarity solutions.
\newblock {\em J. Math. Phys.}, 42(8):3714--3734, 2001.

\bibitem{Catalano_Paola}
D.~Catalano-Ferraioli and P.~Morando.
\newblock Applications of solvable structures to the nonlocal
  symmetry-reduction of {ODE}s.
\newblock {\em J. Nonlinear Math. Phys.}, 16(sup1):27--42, 2009.

\bibitem{Catalano_nonlocal}
D.~Catalano-Ferraioli and P.~Morando.
\newblock Local and nonlocal solvable structures in the reduction of {ODE}s.
\newblock {\em J. Phys. A: Math. Gen.}, 42:035210--035225, 2008.

\end{thebibliography}
\end{document}